\pdfoutput=1
\documentclass[11pt,reqno]{amsart}
\usepackage[utf8]{inputenc}
\usepackage{amsmath,amssymb,amsthm,amsfonts}
\usepackage{mathtools}
\usepackage{bm}

\usepackage[margin=2.5cm]{geometry}

\DeclareMathOperator{\supp}{supp}
\DeclareMathOperator{\scl}{socle}

\DeclareMathOperator{\diam}{diam}
\DeclareMathOperator{\spn}{span}
\DeclarePairedDelimiter{\abs}{\lvert}{\rvert}
\DeclarePairedDelimiter{\norm}{\lVert}{\rVert}
\DeclarePairedDelimiterX{\pair}[2]{\langle}{\rangle}{#1, #2}

\theoremstyle{plain}
\newtheorem{theorem}{Theorem}[section]
\newtheorem{lemma}[theorem]{Lemma}
\newtheorem{proposition}[theorem]{Proposition}
\newtheorem{corollary}[theorem]{Corollary}

\theoremstyle{definition}

\theoremstyle{remark}
\newtheorem{remark}[theorem]{Remark}

\usepackage[shortlabels]{enumitem}
\usepackage{hyperref}

\title[Exact \texorpdfstring{$L^2$}{L2} convergence through wavelets and tent spaces]%
{Exact $L^2$ convergence through the wavelet and tent space characterizations of $h^p(\mathbb{R}^n)$, $0 < p \le 1$}

\author{M. El Aoumari}
\address{Department of Mathematics and Statistics, Concordia University,
Montréal, Québec, Canada}
\email{mohamed.elaoumari@mail.concordia.ca}
\urladdr{https://orcid.org/0009-0003-9984-312X}

\date{\today}


\makeatletter
\g@addto@macro\normalsize{%
  \setlength{\abovedisplayskip}{1.0pt plus 4.55pt}%
  \setlength{\belowdisplayskip}{1.0pt plus 4.55pt}%
}
\makeatother
\begin{document}

\begin{abstract}
We extend the wavelet and tent space characterization of $h^1(\mathbb{R}^n)$ to $0 < p \le 1$ via a local dyadic tent space $t_p$, and derive the unconditional $L^2$-convergence of the resulting atomic decomposition, with an exact error identity refining Ding, Jiang \& Zhu (2019).
\end{abstract}
\maketitle

\section{Introduction}
\subsection{Background and goal}
The local Hardy space $h^p(\mathbb{R}^n)$~\cite{goldberg1979} is the non-homogeneous variant of $H^p(\mathbb{R}^n)$~\cite{feffermanstein1972}, its atoms carrying moment conditions only below unit scale. Meyer's wavelet characterization of $H^1(\mathbb{R}^n)$~\cite[Chapter~5]{meyer1992} proceeds through the dyadic tent space $T_1$~\cite{coifman1985}; the author's thesis~\cite{elaoumari2025} carries this to $h^1(\mathbb{R}^n)$ via a local tent space $t_1$. The present paper extends this to the full range $0 < p \le 1$: we introduce a local dyadic tent space $t_{p}$, related to $T_{p}$ by an exact isometric compression, prove it is a quasi-Banach space with an atomic decomposition (Theorem~\ref{thm:tp-atomic-restate}), and derive a complete wavelet characterization of $h^p(\mathbb{R}^n)$ (Theorem~\ref{thm:hp-wavelet}). The quasi-Banach range forces genuine changes throughout: $\ell^p$-summed coefficients, atomic series that need not converge in $L^1_{\mathrm{loc}}$, pairing conventions for non-Schwartz test objects (Subsection~\ref{ssec:wavelet}), and a Daubechies basis of regularity $r > n_p := n(1/p-1)$ with $N \ge N_p + 1 := \lfloor n_p \rfloor + 1$ vanishing moments, both margins exact.

Among the seven equivalent conditions is a Goldberg-type splitting $f - \varphi * f \in H^{p}$, $\varphi * f \in L^{p}$, running here over admissible Schwartz functions ($\varphi \in \mathcal S(\mathbb{R}^n)$, $\int \varphi = 1$, $\int x^\nu \varphi = 0$ for $0 < \abs{\nu} \le N_p$), via a direct correction lemma (Lemma~\ref{lem:schwartz-correction}) on the classical molecular theory of Taibleson \& Weiss~\cite{taibleson1980}: a single mechanism, sufficient for the full range $0 < p \le 1$ and for the exact $L^2$ identity below, and the reason we fix it here rather than the larger admissible class the splitting also tolerates. At $p = 1$ the thesis~\cite{elaoumari2025} extends this admissible class further, to bounded compactly supported functions, producing atoms rather than molecules; identifying the largest such class, for every $p$, is a separate question from the wavelet characterization pursued here.

A further contribution, and the paper's namesake, is the unconditional $L^2$ convergence of the wavelet-built atomic decomposition of $f \in h^p(\mathbb{R}^n) \cap L^2(\mathbb{R}^n)$ (Proposition~\ref{prop:L2-convergence-E}), sharpening the discrete $L^2$ decomposition of Ding, Jiang \& Zhu~\cite{dingjiangzhu2019}: the error of every finite section is an exact Parseval-tail identity, the atomic partial sums coinciding with orthogonal projections onto disjoint coefficient blocks of an orthonormal wavelet basis.

\subsection{Relation to existing literature}
Wavelet characterizations of the global spaces $H^{p}$, $0 < p \le 1$, are classical: Wu~\cite{wu1992} and García-Cuerva \& Martell~\cite{garciacuerva-martell2001} treat the weighted spaces on $\mathbb{R}$, and on $\mathbb{R}^n$ the characterization follows from $H^p(\mathbb{R}^n) = \dot F^0_{p,2}(\mathbb{R}^n)$ and the $\varphi$-transform of Frazier \& Jawerth~\cite{frazier-jawerth1990}. On local Hardy spaces, wavelets have served chiefly for bilinear decompositions, by Cao, Ky \& Yang~\cite{caoKyYang2018} and Yang, Yuan \& Zhang~\cite{yangYuanZhang2021}, and on spaces of homogeneous type for the characterizations of Fu \& Yang~\cite{fuYang2018} and He, Yang \& Yuan~\cite{heYangYuan2021}. The identification $h^p(\mathbb{R}^n) = F^0_{p,2}(\mathbb{R}^n)$ of Goldberg and Triebel~\cite{goldberg1979, triebel1983}, combined with~\cite{frazier-jawerth1990}, yields in principle a wavelet description of $h^{p}$ through the sequence spaces $f^0_{p,2}$. Our aim is of a different nature: a direct, self-contained, Meyer-style characterization through a local tent space, with explicit atoms, unconditionality, and the Goldberg-type splitting as equivalent conditions, from which the exact $L^2$ identity emerges as a direct consequence, identifying the specific atomic grouping with disjoint coefficient blocks of an orthonormal basis, a structural fact invisible to the abstract sequence space formalism.

\subsection{Conventions}
Throughout, $p \in (0, 1]$ is fixed, and the associated exponents are:

\begin{equation}
\label{eq:exponents}
n_p := n\Bigl( \frac{1}{p} - 1 \Bigr), \quad
  N_p := \lfloor n_p \rfloor, \quad
  (N_p', \delta_p) := \bigl( \lceil n_p \rceil - 1,\;
     n_p - N_p' \bigr)
  =
  \begin{cases}
    \bigl( N_p,\; \{n_p\} \bigr), & n_p \notin \mathbb{Z}, \\[2pt]
    \bigl( N_p - 1,\; 1 \bigr), & n_p \in \mathbb{Z},
  \end{cases}
\end{equation}
the last pair for $p < 1$ only, $\{\cdot\}$ the fractional part, so $\delta_p \in (0, 1]$ and $N_p' + \delta_p = n_p$; $n_p = N_p = 0$ at $p = 1$, $N_p' \le N_p$ with equality unless $n_p \in \mathbb{Z}^+$, and $N_p + 1 > n_p$. We also use, without further comment, $(s+t)^p \le s^p + t^p$, the resulting $p$-triangle inequality, and $\ell^p \hookrightarrow \ell^1$.

A second exponent $q \in (1, \infty]$ indexes atom integrability, with conjugate $q' = q/(q-1)$ ($q' = 1$, $1/q = p/q = 0$ at $q = \infty$); all $L^q, \ell^q$ spaces are over $\mathbb{R}^n$, complex-valued. For $p < 1$, $\norm{\cdot}_p$ is a quasi-norm and $L^p(\mathbb{R}^n) \not\subseteq L^1_{\mathrm{loc}}(\mathbb{R}^n)$.

The wavelet system (Subsection~\ref{ssec:wavelet}) is indexed by $\lambda = k/2^j + e/2^{j+1}$, scale $j \in \mathbb{Z}$, translation $k \in \mathbb{Z}^n$, direction $e \in E$, $k$ also indexing the father translates $\phi_k$ and $e_k$ the single-site father atoms of Section~\ref{sec:tp}; $r$, $N$, $\psi$ denote the basis regularity, vanishing moments, and generating wavelets.

Multi-indices $\nu = (\nu_1, \dots, \nu_n) \in (\mathbb{N} \cup \{0\})^n$ carry the standard notation $\abs{\nu} = \sum \nu_i$, $\nu! = \prod \nu_i!$, $x^\nu = \prod x_i^{\nu_i}$, $\partial^\nu = \prod \partial_{x_i}^{\nu_i}$; secondary multi-indices in binomial expansions are $\beta \le \nu$ componentwise, distinct from the father coefficients $\beta_k$. Hölder exponents are $\delta, \theta \in (0,1]$, moduli $\omega$, $\gamma \in (0,1]$ the measure ratio of $R(\lambda), R_k$, $m$ the window dilation (Subsection~\ref{ssec:wavelet}), $\vartheta \in (0,1)$ the local magnification threshold (Section~\ref{sec:tp}), $\omega_n = \abs{B(0,1)}$, and $\norm{\cdot}_p := \norm{\cdot}_{L^p(\mathbb{R}^n)}$.

Balls are $B = B(x_B, r_B)$, $cB = B(x_B, c r_B)$; for a cube $R$, $cR$ is the concentric dilate of side $c\,\ell(R)$. Dyadic cubes $Q \in \mathcal{D}$ have $\abs{Q} = \ell(Q)^n$. The small/large dichotomy is $r_B \lessgtr 1$ for balls and $\ell(Q) \le 1$ (family $\mathcal{D}^{0}$) for cubes, comparison constants absorbed into $\lesssim$.

The Fourier transform is normalized $\hat f(\xi) = \int f(x)\, e^{-i x \cdot \xi}\, dx$, so $\int f = \hat f(0)$ and $\int x^\nu f = i^{\abs\nu} \partial^\nu \hat f(0)$. Pairings $\pair{f}{g}$ are anti-linear in $g$, equal to $\int f \bar g$ on functions; the duality $\pair{\cdot}{\cdot}_*$ is with respect to $(H^{p})^* \cong \dot\Lambda_{n_p}$ ($\mathrm{H}^1$-$\mathrm{BMO}$ at $p=1$), and the extended pairing conventions \textup{(P1)}-\textup{(P3)} are those of Subsection~\ref{ssec:wavelet}.

$X \lesssim Y$ means $X \le CY$ for a constant depending only on $(n, p, q, m, \text{wavelet basis})$, hence on $(r, N, \gamma, c)$; $X \simeq Y$ means $X \lesssim Y \lesssim X$. Constants may change from line to line.

\section{Preliminaries}
\label{sec:prelim}

\subsection{Dyadic cubes, wavelet basis, pairings, and exhaustions}
\label{ssec:wavelet} Let $\mathcal{D}$ denote the set of all (half-open) dyadic cubes in $\mathbb{R}^n$: $Q \in \mathcal{D}$ if $Q = \prod_{i=1}^n [2^{-j} k_i, 2^{-j}(k_i+1))$ for some $j \in \mathbb{Z}$ and $k = (k_1, \dots, k_n) \in \mathbb{Z}^n$. The side length of $Q$ is $\ell(Q) = 2^{-j}$ and $\abs{Q} = 2^{-jn}$, and we write $\mathcal{D}_{j} := \{Q \in \mathcal{D} : \ell(Q) = 2^{-j}\}$ for the cubes of \emph{generation} $j$, so that $\mathcal{D} = \bigcup_{j \in \mathbb{Z}} \mathcal{D}_{j}$, the union disjoint. We write
\[
  \mathcal{D}^{0} := \{Q \in \mathcal{D} : \ell(Q) \le 1\} = \bigcup_{j \ge 0} \mathcal{D}_{j}
\]
for the family of \emph{small} dyadic cubes. For a (not necessarily dyadic) cube $R$ we write $\mathcal{T}(R) = \{Q \in \mathcal{D} : Q \subseteq R\}$ for the \emph{tent} over $R$ and $\mathcal{T}_{0}(R) = \{Q \in \mathcal{D}^{0} : Q \subseteq R\}$ for its local version.

Given a sequence of scalars $\alpha = \{\alpha(Q)\}_{Q \in \mathcal{D}'}$ indexed by some $\mathcal{D}' \subseteq \mathcal{D}$, we set
\[
  \supp(\alpha) = \{Q \in \mathcal{D}' : \alpha(Q) \ne 0\}, \quad
  \scl(\alpha) = \bigcup_{Q \in \supp(\alpha)} Q
\]
for its \emph{support} and its \emph{socle}, following the terminology of~\cite[Chapter~5, p.~144]{meyer1992} and of~\cite{elaoumari2025}.

We work throughout with a fixed orthonormal wavelet basis of $L^2(\mathbb{R}^n)$ of \emph{tensor-product} compactly supported Daubechies type (see~\cite[Sections~2.A and~4.C]{daubechies1988}). Concretely, let $\psi \in C^r(\mathbb{R})$ be a one-dimensional Daubechies wavelet of order $N$, with compactly supported scaling function $\phi \in C^r(\mathbb{R})$, where $r > n_p$ and $N \ge N_p + 1$. Setting $E = \{0,1\}^n \setminus \{0_{\mathbb{R}^n}\}$, define for each $e \in E$ the tensor-product mother wavelets $\psi^e(x_1, \dots, x_n) = \prod_{i=1}^n \chi^{e_i}(x_i)$, where $\chi^0 = \phi$ and $\chi^1 = \psi$. Set further $\phi(x_1,\dots,x_n) = \prod_{i=1}^n \chi^0(x_i)$, the same letter serving for the one-dimensional factor and its tensor power.

The full wavelet system is then:
\[
  \psi_\lambda(x) = 2^{nj/2}\, \psi^e(2^j x - k), \qquad
  \lambda = \frac{k}{2^j} + \frac{e}{2^{j+1}}, \quad
  k \in \mathbb{Z}^n, j \in \mathbb{Z}, e \in E,
\]
together with the integer translates of the father wavelet $\phi_k(x) = \phi(x - k)$, $k \in \mathbb{Z}^n$.

The index set is:
\[
  \Lambda = \biggl\{ \lambda = \frac{k}{2^j} + \frac{e}{2^{j+1}} \ :\
   k \in \mathbb{Z}^n, j \in \mathbb{Z}, e \in E \biggr\},
   \qquad
  \Lambda^0 = \{\lambda \in \Lambda : j \ge 0\}.
\]
We write $Q(\lambda) \in \mathcal{D}$ for the dyadic cube $\{x \in \mathbb{R}^n : 2^j x - k \in [0,1)^n\}$ ($\lambda \in \Lambda^0$ thus corresponds to $Q(\lambda) \in \mathcal{D}^{0}$). We choose, following~\cite[Chapter~5, p.~142]{meyer1992} as in~\cite[Section~4.2]{elaoumari2025}, by translating the basic wavelets if necessary, a sub-cube $A^e \subseteq [0,1)^n$ with $\abs{A^e} \ge \gamma > 0$ on which $\abs{\psi^e} \ge c > 0$, and likewise a sub-cube $A^0 \subseteq [0,1)^n$ with $\abs{A^0} \ge \gamma$ on which $\abs{\phi} \ge c$, the scaling function being continuous with $\int \phi = 1$. Setting $R(\lambda) = \{x \in \mathbb{R}^n : 2^j x - k \in A^e\}$ and $R_k := k + A^0$, we have $\abs{R(\lambda)} \ge \gamma \abs{Q(\lambda)}$, $\abs{R_k} \ge \gamma$, $\abs{\phi_k} \ge c\, \mathbf{1}_{R_k}$, and:
\begin{equation}
\label{eq:lb-psi}
  \abs{\psi_\lambda(x)} \ge c\, \abs{Q(\lambda)}^{-1/2}\,
  \mathbf{1}_{R(\lambda)}(x), \qquad \lambda \in \Lambda.
\end{equation}

There is a uniform $m \ge 1$ (depending only on the choice of basis) such that:
\[
  \supp \psi_\lambda \subseteq m\,Q(\lambda), \qquad
  \supp \phi_k \subseteq m\,Q_k,
\]
where $Q_k = k + [0,1)^n$, $k \in \mathbb{Z}^n$.

The Daubechies wavelets have, by construction, vanishing moments up to order $N$ (see~\cite[Section~4.C]{daubechies1988}):
\begin{equation}
\label{eq:moments}
  \int_{\mathbb{R}^n} x^\nu\, \psi_\lambda(x)\, dx = 0, \qquad
  \abs\nu \le N, \quad \lambda \in \Lambda.
\end{equation}

The constants implicit in $\lesssim$ may depend on the fixed basis, in particular on $(r, N)$.

Since $\psi_\lambda, \phi_k \in C^r_c(\mathbb{R}^n) \setminus \mathcal S(\mathbb{R}^n)$, the coefficients $\alpha(\lambda) = \pair{f}{\psi_\lambda}$, $\beta_k = \pair{f}{\phi_k}$ need three conventions \textup{(P1)}-\textup{(P3)} for a general $f \in \mathcal{S}'(\mathbb{R}^n)$, agreeing on their common domains (Remark~\ref{rem:pairing-compat}).
\begin{enumerate}[label=\textup{(P\arabic*)}, wide, leftmargin=0pt]
\item
\label{pair:general} Every $f \in \mathcal{S}'(\mathbb{R}^n)$ satisfies, for some $N_0, K \ge 0$, $C > 0$ (\cite[Section~2.3]{Grafakos2014}):
\begin{equation}
\label{eq:order}
\abs{\pair{f}{\zeta}} \;\le\; C \sum_{\abs{\beta} \le N_0} \sup_{x \in \mathbb{R}^n} (1 + \abs{x})^{K}\, \abs{\partial^\beta \zeta(x)}, \qquad \zeta \in \mathcal S(\mathbb{R}^n),
\end{equation}
and has \emph{order less than $r$} if this holds with $N_0 < r$, so that $\psi_\lambda, \phi_k \in C^{N_0}_c(\mathbb{R}^n)$. For such $f$ and $\eta \in C^{N_0}_c(\mathbb{R}^n)$, with $\theta_\varepsilon$ any standard mollifier, set:
\begin{equation}
\label{eq:extended-pairing}
\pair{f}{\eta} \;:=\; \lim_{\varepsilon \to 0}\, \pair{f}{\eta * \theta_\varepsilon},
\end{equation}
well defined and mollifier-independent since $\partial^\beta(\eta * \theta_\varepsilon) \to \partial^\beta \eta$ uniformly on a fixed compact set, and~\eqref{eq:order} applies to differences; it agrees with $\int f\bar\eta$ for $f \in L^1_{\mathrm{loc}}$, and, by the same estimate applied uniformly in $M$:
\begin{equation}
\label{eq:P1-continuity}
  f_M \to f \text{ in } \mathcal{S}'(\mathbb{R}^n) \text{ with~\eqref{eq:order} uniform in } M
  \;\implies\; \pair{f_M}{\eta} \to \pair{f}{\eta}, \quad \eta \in C^{N_0}_c(\mathbb{R}^n).
\end{equation}
Under \textup{(P1)}: $\alpha(\lambda) := \pair{f}{\psi_\lambda}$, $\beta_k :=\pair{f}{\phi_k}$, via~\eqref{eq:extended-pairing}.

\item
\label{pair:atomic} If $f = \sum_j \mu_j a_j$ in $\mathcal{S}'(\mathbb{R}^n)$, with $L^q$ $h^{p}$-atoms $a_j$ and $\sum_j \abs{\mu_j}^p < \infty$, set:
\begin{equation}
\label{eq:atomic-pairing}
\alpha(\lambda) := \sum_j \mu_j \pair{a_j}{\psi_\lambda}, \qquad \beta_k := \sum_j \mu_j \pair{a_j}{\phi_k},
\end{equation}
absolutely convergent, term-wise and in $j$: $\sup_j\abs{\pair{a_j}{\eta}} \le C_\eta$ by Hölder's Inequality for large atoms and by the Taylor/moment estimate~\eqref{eq:taylor-holder} for small ones, against $\sum_j\abs{\mu_j} < \infty$. Independent of the decomposition (Remark~\ref{rem:pairing-compat}).

\item
\label{pair:conditionG} If $g := f - \varphi_0 * f \in H^p(\mathbb{R}^n)$ and $u := \varphi_0 * f \in L^p(\mathbb{R}^n)$, $\varphi_0$ the fixed function of~\eqref{eq:admiss-p-s}, set:
\begin{equation}
\label{eq:G-pairing}
\alpha(\lambda) := \pair{g}{\psi_\lambda}_{*} + \int_{\mathbb{R}^n} u\, \overline{\psi_\lambda}\, dx, \qquad \beta_k := \pair{g}{\phi_k}_{*} + \int_{\mathbb{R}^n} u\, \overline{\phi_k}\, dx,
\end{equation}
$\pair{\cdot}{\cdot}_{*}$ the $H^{p}$-$\dot\Lambda_{n_p}$ duality ($\mathrm H^1$-$\mathrm{BMO}$ at $p=1$), applicable since $\psi_\lambda, \phi_k \in C^{N_p,\delta}_c \hookrightarrow \dot\Lambda_{n_p}$, the integrals absolute since $u$ is continuous. Independent of the admissible Schwartz function used (Remark~\ref{rem:pairing-compat}).
\end{enumerate}

\begin{remark}\label{rem:pairing-compat}
The conventions agree throughout. \textup{(P1)}$\cap$\textup{(P2)}: apply~\eqref{eq:P1-continuity} to the atomic partial sums. \textup{(P2)}$\cap$\textup{(P3)}: on continuous $h \in H^{p}$ the duality pairing $\pair{h}{\eta}_*$ coincides with $\int h\bar\eta$, since band-limited mollifications $\psi_t * h \to h$ in $H^{p}$ and locally uniformly, by the Nikol'skij Inequality~\cite[Section~1.3.2]{triebel1983}, giving both decomposition-independence in \textup{(P2)} and $\varphi$-independence in \textup{(P3)}, the latter via $h := \varphi_1 * f - \varphi_2 * f \in H^{p} \cap C^\infty$.
\end{remark}

An important result is the inhomogeneous wavelet decomposition (see~\cite[Chapter~3]{meyer1992}, in particular pages 109 and 110): for every $f \in \mathcal{S}'(\mathbb{R}^n)$ of order less than $r$ (so that the coefficients below are defined, cf.\ convention \textup{(P1)}), we have:
\begin{equation}
\label{eq:inh-wavelet-decomp}
f = \sum_{k \in \mathbb{Z}^{n}} \pair{f}{\phi_k}\, \phi_{k} + \sum_{\lambda \in \Lambda^0} \pair{f}{\psi_\lambda}\, \psi_{\lambda} = \sum_{k \in \mathbb{Z}^{n}} \beta_k\, \phi_{k} + \sum_{\lambda \in \Lambda^0} \alpha(\lambda)\, \psi_{\lambda},
\end{equation}
with convergence in $\mathcal{S}'(\mathbb{R}^n)$, summably in the sense recalled below, the series of pairings against each $\zeta \in \mathcal S$ converging absolutely.

Following~\cite[Chapter~4]{elaoumari2025}, an \emph{exhaustion} of a countable index set $I$ is an increasing sequence $\{F_m\}_{m \in \mathbb{N}}$ of finite subsets with union $I$. The index sets of this paper, $\mathcal{D}^{0}$, $\mathbb{Z}^n$ and $\Lambda^0$, are countable, and any natural choice of exhaustion will do.

The sharpest frame for the convergence statements of this paper is the net of finite subsets of $I$, directed by inclusion: a family $(x_i)_{i \in I}$ in a quasi-normed space $X$ is \emph{summable} to $S$ when the net of partial sums $\bigl( \sum_{i \in F} x_i \bigr)_F$ converges to $S$. Every atomic and wavelet series below converges in this mode, and the following lemma records its equivalence with the exhaustion formulations and the criterion through which we verify it.

\begin{lemma}
\label{lem:net-summability} Let $X$ be a complete quasi-normed space whose quasi-norm is $p$-sub-additive for some $p \in (0, 1]$, let $I$ be a countable index set, and let $(x_i)_{i \in I} \subseteq X$. The following are equivalent:
\begin{enumerate}[(i)]
\item the family is summable: the net of finite partial sums converges in $X$,
\item for every $\varepsilon > 0$ there is a finite $F_0 \subseteq I$ with $\norm[\big]{\sum_{i \in G} x_i}_X < \varepsilon$ for every finite $G \subseteq I$ disjoint from $F_0$,
\item for every exhaustion $\{F_m\}$ of $I$, the partial sums $\sum_{i \in F_m} x_i$ converge in $X$, with a limit independent of the exhaustion.
\end{enumerate}
In this case the three limits coincide. In particular, if $\displaystyle \norm[\big]{\sum_{i \in G} x_i}_X^p \le C \sum_{i \in G} c_i$ for every finite $G \subseteq I$, with $c_i \ge 0$ and $\sum_{i \in I} c_i < \infty$, then the family is summable.
\end{lemma}

\begin{proof}
For Banach spaces these equivalences are classical, see Lindenstrauss \& Tzafriri~\cite[Proposition~1.c.1]{lindenstrausstzafriri1977}, and the arguments involve only the invariant metric $d(x, y) := \norm{x - y}_X^p$ and completeness, so they transfer verbatim: (i) implies (ii) by directedness, (ii) implies (i) and (iii) since the net is then Cauchy, and if (ii) fails at some $\varepsilon$, finite blocks $G_k$ with $\norm{\sum_{i \in G_k} x_i}_X \ge \varepsilon$, chosen inductively disjoint from $E_{k+1} \cup G_1 \cup \dots \cup G_{k-1}$ for a reference exhaustion $\{E_m\}$, interlace into an exhaustion whose partial sums are not Cauchy, contradicting (iii). The criterion gives (ii) directly, choosing $F_0$ with $\sum_{i \notin F_0} c_i < \varepsilon^p / C$.
\end{proof}

Two consequences are used throughout. For $c_i \ge 0$ the net converges to the supremum, $\sum_{i \in I} c_i := \lim_m \sum_{i \in F_m} c_i = \sup_F \sum_{i \in F} c_i \in [0, \infty]$ along any exhaustion, and likewise for the increasing partial square functions $\bigl( \sum_{\lambda \in F} \abs{\alpha(\lambda)}^2 \abs{\psi_\lambda(x)}^2 \bigr)^{1/2} \nearrow S_{\mathrm{Loc}}^{\psi}(\alpha)(x) := \bigl( \sum_{\lambda \in \Lambda^0} \abs{\alpha(\lambda)}^2 \abs{\psi_\lambda(x)}^2 \bigr)^{1/2}$ and their $L^p$-quasi-norms, so that the supremum over finite subsets in condition (A) of Theorem~\ref{thm:hp-wavelet} may be computed along any exhaustion. And every atomic and wavelet series in this paper obeys $\norm{\sum_{i \in F} \mu_i a_i}_X^p \le C \sum_{i \in F} \abs{\mu_i}^p$ over finite sets with $\ell^p$-summable coefficients, hence is summable by the criterion of Lemma~\ref{lem:net-summability}, and sums such as $\sum_{\lambda \in \Lambda^0} \alpha(\lambda) \psi_\lambda$ or $\sum_j \mu_j a_j$ are understood in that sense, computable along any exhaustion.

\subsection{The local Hardy space \texorpdfstring{$h^p(\mathbb{R}^n)$}{hp(Rn)}}
\label{ssec:lhp} We recall Goldberg's atomic characterization~\cite{goldberg1979}. For $0 < p \le 1$, a measurable function $a \colon \mathbb{R}^n \to \mathbb{C}$ is called an \emph{$L^q$ $h^{p}$-atom associated with $B$} (for $1 < q \le \infty$), where $B$ is a ball of radius $r_0$ (resp.\ a cube of side length $\ell_0$), if:
\begin{enumerate}[(i)]
\item $\supp a \subseteq B$,
\item $\norm{a}_q \le \abs{B}^{1/q - 1/p}$ \quad (with $1/q = 0$ when $q = \infty$),
\item if $r_0 < 1$ (resp.\ $\ell_0 \leq 1$), then $\int x^\nu\, a(x)\, dx = 0$ for all $0 \le \abs{\nu} \le N_p$ and no vanishing moments are required if $r_0 \ge 1$ (resp.\ $\ell_0 > 1$).
\end{enumerate}
A tempered distribution $f$ belongs to $h^p(\mathbb{R}^n)$ if and only if it admits an atomic decomposition: \( f = \sum_j \lambda_j a_j, \) converging in $\mathcal{S}'(\mathbb{R}^n)$, where each $a_j$ is an $L^q$ $h^{p}$-atom and $\sum_j \abs{\lambda_j}^p < \infty$. The atomic quasi-norm:
\[
\norm{f}_{h^{p},\mathrm{at}} = \inf\biggl\{\Bigl(\sum_j \abs{\lambda_j}^p \Bigr)^{1/p}\biggr\},
\]
(the infimum taken over all such decompositions), is independent of $q$, and is equivalent to several other quasi-norms on $h^{p}$, including the local maximal-function quasi-norm:
\[
\norm{f}_{h^{p},\max} = \norm[\Big]{\sup_{0 < t < 1} \abs{\varphi_t * f}}_p,
\]
for any fixed $\varphi \in \mathcal{S}(\mathbb{R}^n)$ with $\int \varphi = 1$, where $\varphi_t(x) = t^{n}\varphi(x/t)$. See~\cite[Theorem~1, Lemma~5]{goldberg1979}.

Atoms enjoy two uniform bounds. First, there is a constant $C_{\mathrm{at}} = C_{\mathrm{at}}(n, p, q)$ such that every $L^q$ $h^{p}$-atom $a$, $1 < q \le \infty$, associated with a ball or a cube $B$, satisfies:
\begin{equation}
\label{eq:atom-hp-bound}
\norm{a}_{h^{p}, \max} \;\le\; C_{\mathrm{at}} ,
\end{equation}
second, by Hölder's Inequality on $B$ and the size condition:
\begin{equation}
\label{eq:atom-l1}
  \norm{a}_1
  \;\le\; \norm{a}_q\, \abs{B}^{1/q'}
  \;\le\; \abs{B}^{1 - 1/p} .
\end{equation}

\subsection{The real Hardy space \texorpdfstring{$H^p(\mathbb{R}^n)$}{Hp(Rn)}}
\label{ssec:HP} For $0 < p \le 1$, the space $H^p(\mathbb{R}^n)$ consists of all $f \in \mathcal{S}'(\mathbb{R}^n)$ such that:
\[
\norm{f}_{H^{p}, \max} := \norm[\Big]{\sup_{t > 0}\, \abs{\varphi_t * f}}_p < \infty,
\]
for some (equivalently, any) fixed $\varphi \in \mathcal S(\mathbb{R}^n)$ with $\int \varphi \ne 0$. Distinct choices of $\varphi$, as well as the grand maximal function, yield equivalent quasi-norms~\cite[Theorem~11]{feffermanstein1972}, \cite[Chapter~III]{stein1993}. In contrast with $h^{p}$, the supremum runs over \emph{all} scales $t > 0$, and we have the strict continuous inclusion $H^p(\mathbb{R}^n) \subsetneq h^p(\mathbb{R}^n)$.

For $1 < q \le \infty$, a measurable function $a$ is an \emph{$L^q$ $H^{p}$-atom associated with the ball or cube $B$} if:
\begin{enumerate}[(i)]
\item $\supp(a) \subseteq B$,
\item $\norm{a}_q \le \abs{B}^{1/q - 1/p}$ \quad (with $1/q = 0$ when $q = \infty$),
\item $\int x^\nu\, a(x)\, dx = 0$ for all $0 \leq \abs{\nu} \leq N_p$.
\end{enumerate}
The definition differs from its local counterpart only in \textup{(iii)}: the vanishing moments are required at \emph{every} scale. A tempered distribution $f$ belongs to $H^p(\mathbb{R}^n)$ if and only if it admits a decomposition $f = \sum_j \lambda_j a_j$, converging in $\mathcal{S}'(\mathbb{R}^n)$, into $L^q$ $H^{p}$-atoms with $\sum_j \abs{\lambda_j}^p < \infty$, and the atomic quasi-norm:
\[
\norm{f}_{H^{p},\mathrm{at}} = \inf \biggl\{ \Bigl( \sum_j \abs{\lambda_j}^p \Bigr)^{1/p}\biggr\},
\]
(the infimum over all such decompositions), is independent of $q$ up to equivalence and equivalent to $\norm{\cdot}_{H^{p}, \max}$. The case $q = \infty$ ($n = 1$) is due to Coifman~\cite{coifman1974}, the general case to Latter~\cite{latter1978}. See also~\cite[Chapter~III, Theorem~2]{stein1993} and Taibleson \& Weiss~\cite{taibleson1980}.

Molecules relax the support condition to quantitative decay. Given $\varepsilon > 0$ and a ball $B = B(x_0, r)$, write $A_0(B) = 2B$ and $A_k(B) = 2^{k+1} B \setminus 2^k B$, $k \ge 1$. A function $m \in L^q(\mathbb{R}^n)$ is a \emph{$(p, q, \varepsilon)$-molecule centered at $B$} if:
\begin{enumerate}[(i)]
\item $\norm{m}_{L^q(A_k(B))} \le 2^{-k\varepsilon} \abs{2^k B}^{1/q - 1/p}$ for all $k \ge 0$
\item $\int x^\nu\, m(x)\, dx = 0$ for all $0 \leq \abs{\nu} \le N_p$ (the integrals converging absolutely).
\end{enumerate}
Every $L^q$ $H^{p}$-atom is a $(p,q,\varepsilon)$-molecule for every $\varepsilon > 0$, and conversely there is a constant $C_{\mathrm{mol}} = C_{\mathrm{mol}}(n, p, q, \varepsilon)$ such that every $(p,q,\varepsilon)$-molecule $m$ lies in $H^{p}$ with $\norm{m}_{H^{p}} \le C_{\mathrm{mol}}$~\cite{taibleson1980}. So $f \in H^{p}$ if and only if $f = \sum_j \lambda_j m_j$ ($\mathcal{S}'(\mathbb{R}^n)$-convergence) with $(p,q,\varepsilon)$-molecules $m_j$ and $\sum_j \abs{\lambda_j}^p < \infty$, the corresponding quasi-norm being equivalent to $\norm{\cdot}_{H^{p}, \max}$. Molecular decompositions are the natural output of operators that destroy compact support, such as the convolution corrections with Schwartz functions covered in Lemma~\ref{lem:schwartz-correction} below.

Two facts used throughout, for both $h^{p}$ and $H^{p}$, are recorded here. For $\varphi \in C^{k,\delta}(\mathbb{R}^n)$, write $[\partial^\nu \varphi]_\delta := \sup_{x \ne y} \abs{\partial^\nu\varphi(x) - \partial^\nu\varphi(y)}\, \abs{x-y}^{-\delta}$ for the Hölder semi-norm of a $k$-th order derivative, $\abs{\nu} = k$. The single property of this regularity used throughout is the Hölder form of the Taylor remainder: with $P_{x_0}$ the Taylor polynomial of degree $k$ of $\varphi$ at $x_0$:
\begin{equation}
\label{eq:taylor-holder}
\abs{\varphi(x) - P_{x_0}(x)}
\;\le\; C_{n, k}
   \Bigl( \max_{\abs{\nu} = k}
   [\partial^\nu \varphi]_{\delta} \Bigr)\,
   \abs{x - x_0}^{k + \delta},
\qquad x \in \mathbb{R}^n .
\end{equation}

An explicit admissible band-limited Schwartz function is exhibited as follows: fix $\eta \in C^\infty_c(\mathbb{R}^n)$ radial and real-valued with $\eta \equiv 1$ on $\{\abs{\xi} \le 1\}$ and $\supp(\eta) \subseteq \{\abs{\xi} \le 2\}$, and set $\varphi_0 := \mathcal F^{-1}(\eta) \in \mathcal S(\mathbb{R}^n)$ (real-valued, since $\eta$ is real and even). Then:
\[
\int_{\mathbb{R}^n} \varphi_0 = \hat\varphi_0(0) = 1, \qquad \int_{\mathbb{R}^n} x^\nu\, \varphi_0(x)\, dx = i^{\abs\nu}\, \partial^\nu \hat\varphi_0(0) = 0 \quad \text{for all } \abs{\nu} \ge 1,
\]
since $\hat\varphi_0 = \eta$ is constant near the origin, so that $\varphi_0$ is an \emph{admissible Schwartz function}, that is, $\int \varphi_0 = 1$ with $\int x^\nu \varphi_0 = 0$ for $0 < \abs{\nu} \le N_p$, and:
\begin{equation}
\label{eq:admiss-p-s}
\varphi_0 \;\in\; \mathcal{S}(\mathbb{R}^n), \qquad \supp(\hat \varphi_0) \subseteq \{\abs{\xi} \le 2\} \text{ is compact} .
\end{equation}

\subsection{The homogeneous dyadic tent space \texorpdfstring{$T_{p}$}{Tp}}
For context and contrast, we recall the homogeneous dyadic tent space $T_{p}$, the discrete model of the tent spaces of Coifman, Meyer \& Stein~\cite{coifman1985}. Given a sequence $\alpha = \{\alpha(Q)\}_{Q \in \mathcal{D}}$, define its homogeneous square function by:
\begin{equation}
\label{eq:S-p}
S(\alpha)(x) := \biggl( \sum_{\substack{Q \in \mathcal{D} \\ x \in Q}} \abs{\alpha(Q)}^2\, \abs{Q}^{-1} \biggr)^{1/2}.
\end{equation}
The space $T_{p}$ consists of the sequences with $S(\alpha) \in L^p(\mathbb{R}^n)$, with quasi-norm $\norm{\alpha}_{T_{p}} = \norm{S(\alpha)}_p$. The dyadic model coincides, up to the renormalization $\alpha(Q) = \abs{Q}^{1/2} s_Q$, with the Frazier-Jawerth sequence space $\dot f^{0}_{p,2}$~\cite{frazier-jawerth1990}, whose atomic decomposition is due to Frazier \& Jawerth: a \emph{$T_{p}$-atom} is a sequence supported in a tent $\mathcal{T}(R)$, $R \in \mathcal{D}$, with $\sum_Q \abs{\alpha(Q)}^2 \le \abs{R}^{1 - 2/p}$, atoms have quasi-norm $\lesssim 1$, and every $\alpha \in T_{p}$ is $\sum_j \lambda_j \alpha_j$ with atoms $\alpha_j$ and $\sum_j \abs{\lambda_j}^p < \infty$~\cite[Sections~5--6]{frazier-jawerth1990}. The wavelet coefficient sequence of an element of $H^p(\mathbb{R}^n)$ belongs to $T_{p}$ and conversely, Meyer's characterization of $H^1$ at $p = 1$~\cite[Chapter~5, Theorem~1, p.~143]{meyer1992} and, for $0 < p \le 1$, a consequence of $H^{p} = \dot F^{0}_{p,2}$ and the $\varphi$-transform~\cite{frazier-jawerth1990}. Our space $t_{p}$ is the non-homogeneous, small cube analogue of $T_{p}$, whose structure theory parallels that of the inhomogeneous $f^{0}_{p,2}$ but is developed directly in the tent space language, in parallel with~\cite{elaoumari2025}.

\section{The local dyadic tent space \texorpdfstring{$t_{p}$}{tp}}
\label{sec:tp}

The structure of this section mirrors~\cite[Section~4.1]{elaoumari2025}, the modifications required for $p < 1$ pointed out as they arise. To the best of our knowledge, $t_{p}$ appears here for the first time for $p < 1$, its case $p = 1$ being the space $t_1$ of the thesis, and its atomic theory is the first for a local dyadic tent space in the quasi-Banach range.

\subsection{Definition and basic properties}
The space carries two layers of scalar coefficients on disjoint index copies: a \emph{father layer} $\beta = \{\beta_k\}_{k \in \mathbb{Z}^n}$ indexed by the unit lattice, whose cubes $Q_k := k + [0, 1)^n$ tile $\mathbb{R}^n$, and a \emph{mother layer} $\alpha = \{\alpha(Q)\}_{Q \in \mathcal{D}^{0}}$ indexed by the small dyadic cubes, the unit cubes of $\mathcal{D}^{0}$ carrying wavelet coefficients of their own alongside the lattice copy. For a pair $(\beta, \alpha)$, we define its \emph{local square function} by:
\begin{equation}\label{eq:Sz-p}
S_{\mathrm{loc}}(\beta, \alpha)(x) := \biggl( \sum_{k \,:\, x \in Q_k} \abs{\beta_k}^2 \;+\; \sum_{x \in Q \in \mathcal{D}^{0}} \abs{\alpha(Q)}^2\, \abs{Q}^{- 1} \biggr)^{1/2},
\end{equation}
the father terms self-normalized since $\abs{Q_k} = 1$. A pair supported on a single layer is identified with the corresponding sequence, and we write $S_{\mathrm{loc}}(\alpha) := S_{\mathrm{loc}}(0, \alpha)$ and $S_{\mathrm{loc}}(\beta) := S_{\mathrm{loc}}(\beta, 0)$ accordingly, likewise for the truncated variants below. The mother layer of~\eqref{eq:Sz-p} matches, at $p = 1$, the homogeneous size functional of~\cite[Chapter~5, p.~145]{meyer1992} in the local form of~\cite[Definition~4.1.3]{elaoumari2025}, the father layer, which the thesis and the classical theory carry as a separate $\ell^p$ clause, being incorporated into the space itself in the present extended form.

The \emph{local dyadic tent space} $t_{p}$ is then defined as the set of pairs $(\beta, \alpha)$ with $S_{\mathrm{loc}}(\beta, \alpha) \in L^p(\mathbb{R}^n)$, equipped with the quasi-norm \(\norm{(\beta, \alpha)}_{t_{p}} := \norm{S_{\mathrm{loc}}(\beta, \alpha)}_p\). The unit cubes tiling $\mathbb{R}^n$ disjointly, each point sees exactly one father term, so the father layer sits inside $t_{p}$ isometrically: $\norm{(\beta, 0)}_{t_{p}}^p = \sum_{k \in \mathbb{Z}^n} \abs{\beta_k}^p$, an exact identity.

When $p = 1$, $\norm{\cdot}_{t_{p}}$ is a genuine norm, but when $p < 1$, it is a $p$-norm, and the corresponding topology is metrizable by $d(a, b) = \norm{a - b}_{t_{p}}^p$.

\begin{proposition}
\label{prop:tp-quasi-norm} $(t_{p}, \norm{\cdot}_{t_{p}})$ is a quasi-normed vector space. When $p = 1$, it is a normed vector space.
\end{proposition}

\begin{proof}
Vector space structure and absolute homogeneity are immediate. The $p$-triangle inequality is obtained through integration, $p$-sub-additivity and the triangle inequality for the $\ell^2$-sum in the square function.
\end{proof}

\begin{lemma}
\label{lem:density-tp} $\displaystyle \ell^2_{\mathrm{finite}}(\mathbb{Z}^n \sqcup \mathcal{D}^{0}) := \bigcup_{\substack{F \subseteq \mathbb{Z}^n \sqcup \mathcal{D}^{0} \\ \text{finite}}} \ell^2(F)$, the finitely supported pairs, is dense in $t_{p}$.
\end{lemma}

\begin{proof}
The father layer is immediate: by the exact identity, $\norm{(\beta, 0) - (\beta \mathbf{1}_F, 0)}_{t_{p}}^p = \sum_{k \notin F} \abs{\beta_k}^p \to 0$ along the net of finite subsets, the tail of a convergent series, and the two layers combine by $p$-sub-additivity. It remains to treat the mother layer.
Notice first that Monotone Convergence tells us that for a sequence $\alpha$, we have: $\norm{S_{\mathrm{loc}}(\alpha)}_2 = \norm{\alpha}_{\ell^2(\mathcal{D}^{0})}$. This, alongside Hölder’s Inequality, gives us for any finite subset $F$ of $\mathcal{D}^{0}$ and any $\alpha \in \ell^2(F)$:
\begin{equation}
\label{eq:tp-norm}
\biggl(\inf_{Q \in F} \abs*{Q} \biggr)^{1 - p/2} \, \norm*{\alpha}_{\ell^2(F)}^p \leq \norm*{(0, \alpha)}_{t_{p}}^p \leq \biggl(\sum_{Q \in F} \abs*{Q} \biggr)^{1 - p/2} \, \norm*{\alpha}_{\ell^2(F)}^p.
\end{equation}
The upper bound is Hölder's Inequality on the socle at the exponent $2/p$, and the lower bound follows by comparing $\int S_{\mathrm{loc}}(\alpha)^2 = \norm{\alpha}_{\ell^2(F)}^2$ with $\norm{S_{\mathrm{loc}}(\alpha)}_\infty^{2 - p} \int S_{\mathrm{loc}}(\alpha)^p$ through the pointwise bound $S_{\mathrm{loc}}(\alpha) \le \norm{\alpha}_{\ell^2(F)} \bigl( \inf_{Q \in F} \abs{Q} \bigr)^{-1/2}$. Given $(0, \alpha) \in t_{p}$ and an exhaustion $\{F_m\}_{m \in \mathbb{N}}$ of $\mathcal{D}^{0}$, the sequence $\alpha \mathbf{1}_{F_m}$ satisfies $S_{\mathrm{loc}}(\alpha \mathbf{1}_{F_m}) \nearrow S_{\mathrm{loc}}(\alpha)$ pointwise (monotone increase as the inner sum involves more terms), and is dominated by $S_{\mathrm{loc}}(\alpha) \in L^p$. Dominated convergence in $L^p$ for $0 < p \le 1$ gives $\norm{S_{\mathrm{loc}}(\alpha) - S_{\mathrm{loc}}(\alpha \mathbf{1}_{F_m})}_p \to 0$, which gives $\norm{(0, \alpha - \alpha \mathbf{1}_{F_m})}_{t_{p}} \to 0$, identically to the proof of~\cite[Lemma~4.1.1]{elaoumari2025}, alongside the approximation step:
\[
\norm{(0, \alpha - \alpha \mathbf{1}_{F_m})}_{t_{p}}^p \le \sqrt{2} \, \norm{(0, \alpha)}_{t_{p}}^{p/2} \, \norm{S_{\mathrm{loc}}(\alpha) - S_{\mathrm{loc}}(\alpha \mathbf{1}_{F_m})}_p^{p/2}. \qedhere
\]
\end{proof}

\begin{theorem}
\label{thm:tp-quasi-banach} $(t_{p}, \norm{\cdot}_{t_{p}})$ is a quasi-Banach space. At $p = 1$, it is a Banach space.
\end{theorem}

\begin{proof}
We mirror the homogeneous argument of~\cite[Chapter~5]{meyer1992}, in the local form of~\cite[Theorem~4.1.1]{elaoumari2025}, with two modifications. First, the completeness of $\ell^2(F)$ remains a finite dimensional fact independent of $p$, so a Cauchy sequence of pairs $\{a^{(j)}\}$ in $t_{p}$ restricted to any finite $F \subseteq \mathbb{Z}^n \sqcup \mathcal{D}^{0}$ is Cauchy in $\ell^2(F)$ and converges to some $a_F \in \ell^2(F)$. Compatibility across an exhaustion gives a well-defined limit pair $a$. Second, the final $\varepsilon/3$-style argument uses the $p$-triangle inequality instead of the ordinary triangle inequality:
\[
\norm{a - a^{(j)}}_{t_{p}}^p \le \norm{a - a\, \mathbf{1}_{F_m}}_{t_{p}}^p + \norm{a\, \mathbf{1}_{F_m} - a^{(j)} \mathbf{1}_{F_m}}_{t_{p}}^p + \norm{a^{(j)} \mathbf{1}_{F_m} - a^{(j)}}_{t_{p}}^p. \qedhere
\]
\end{proof}

\subsection{The compression \texorpdfstring{$T_{p} \to t_{p}$}{Tp to tp}}
For $\alpha \in T_{p}$ and $x \in Q_k := k + [0, 1)^n$, the large scale part of the size functional~\eqref{eq:S-p} runs over the cubes $Q$ with $x \in Q$ and $\ell(Q) > 1$, which are exactly the dyadic ancestors of $Q_k$, the same chain for every $x \in Q_k$: the large scale mass is constant on each unit cube. We may thus compress it into a single father coefficient per cell, defining the \emph{compression} $\Theta \colon T_{p} \to t_{p}$ by:
\[
\Theta \alpha \;:=\; \bigl( \beta(\alpha),\, \alpha|_{\mathcal{D}^{0}} \bigr),
\qquad
\beta(\alpha)_k \;:=\; \Bigl( \sum_{Q \supsetneq Q_k} \abs{\alpha(Q)}^2\, \abs{Q}^{-1} \Bigr)^{1/2},
\]
the sum over the ancestors of $Q_k$, consistent with the identity $\sum_{k \in \mathbb{Z}^n} \abs{\beta_k}^2 = \sum_{\lambda \in \Lambda \setminus \Lambda^0} \abs{\alpha(\lambda)}^2$ of~\cite[Chapter~3, Section~10, p.~112]{meyer1992}, there derived from $V_0$ being the Hilbert direct sum of the spaces $W_j$, $j < 0$: the exchange of the large scale mass for father layer mass, which the compression refines from a global identity to one per unit cell. We call a father sequence $\beta$ \emph{sibling-constant} if $\beta_k = \beta_{k'}$ whenever $Q_k$ and $Q_{k'}$ have the same dyadic parent.

\begin{proposition}
\label{prop:Tp-tp-embedding} Let $0 < p \le 1$.
\begin{enumerate}[(i)]
\item For every $\alpha \in T_{p}$ and every $x \in \mathbb{R}^n$, $S(\alpha)(x) = S_{\mathrm{loc}}(\Theta\alpha)(x)$. In particular $\norm{\Theta\alpha}_{t_{p}} = \norm{\alpha}_{T_{p}}$: the compression is an exact isometry.
\item The image of $\Theta$ is the set of pairs $(\beta, \alpha^\circ) \in t_{p}$ whose father layer is nonnegative and sibling-constant. On the closed subspace of sibling-constant pairs, the \emph{parent-scale realization map} $J$, defined by $\alpha|_{\mathcal{D}^{0}} := \alpha^\circ$, $\alpha(P) := 2^{n/2}\, \beta_{k(P)}$ for each dyadic cube $P$ of side $2$, with $k(P)$ any child cell of $P$, and $\alpha(Q) := 0$ for $\ell(Q) > 2$, is a linear isometry into $T_{p}$ with $\Theta J(\beta, \alpha^\circ) = (\abs{\beta}, \alpha^\circ)$. In particular $\Theta J$ is the identity on the image of $\Theta$.
\item Forgetting the father layer, $\norm{(0, \alpha|_{\mathcal{D}^{0}})}_{t_{p}} \le \norm{\alpha}_{T_{p}}$, with equality if and only if $\alpha$ vanishes on the large scales, and on such sequences $\alpha \mapsto (0, \alpha|_{\mathcal{D}^{0}})$ is an isometry onto the mother layer of $t_{p}$.
\end{enumerate}
\end{proposition}

\begin{proof}
For \textup{(i)}, fix $x \in Q_k$. By the ancestor observation and the self-normalization of the father terms:
\[
S(\alpha)(x)^2
\;=\; \sum_{Q \supsetneq Q_k} \abs{\alpha(Q)}^2\, \abs{Q}^{-1}
\;+ \sum_{x \in Q \in \mathcal{D}^{0}} \abs{\alpha(Q)}^2\, \abs{Q}^{-1}
\;=\; S_{\mathrm{loc}}(\Theta\alpha)(x)^2 ,
\]
the first sum being $\beta(\alpha)_k^2$ by definition, and taking $L^p$-quasi-norms gives the isometry. For \textup{(ii)}, the ancestor chain of $Q_k$ is determined by its dyadic parent, so two sibling cells share it and $\beta(\alpha)$ is nonnegative and sibling-constant. Conversely, for a sibling-constant pair, the sequence $J(\beta, \alpha^\circ)$ has, at $x \in Q_k$, a single large scale term, $\abs{\alpha(P)}^2 \abs{P}^{-1} = 2^{n} \beta_k^2 \cdot 2^{-n} = \beta_k^2$, so $S(J(\beta, \alpha^\circ)) = S_{\mathrm{loc}}(\abs{\beta}, \alpha^\circ) = S_{\mathrm{loc}}(\beta, \alpha^\circ)$ pointwise, the size functionals seeing only $\abs{\beta_k}$. This gives at once the isometry of $J$, the identity $\Theta J(\beta, \alpha^\circ) = (\abs{\beta}, \alpha^\circ)$, and, for $\beta \ge 0$, the realization of the pair, and $J$ is clearly linear. For \textup{(iii)}, $S_{\mathrm{loc}}(0, \alpha|_{\mathcal{D}^{0}}) \le S_{\mathrm{loc}}(\Theta\alpha)$ pointwise by solidity, with equality of the $L^p$-quasi-norms exactly when $\alpha(Q) = 0$ for every $\ell(Q) > 1$, where \textup{(i)} reads $S(\alpha) = S_{\mathrm{loc}}(0, \alpha|_{\mathcal{D}^{0}})$ pointwise.
\end{proof}

The compression realizes $t_{p}$ as the local form of $T_{p}$ with nothing discarded, by \textup{(i)}: small scales pass unchanged and large scales compress losslessly to scale one, mirroring Goldberg's splitting of $f \in h^{p}$ into $f - \varphi*f \in H^{p}$ and the unit-scale, cancellation-free $\varphi*f \in L^p(\mathbb{R}^n)$. By \textup{(ii)}, $t_{p}$ carries strictly finer scale-one information, its father layer varying freely rather than being constant on parent blocks. We develop the structure theory of $t_{p}$ directly rather than transporting it through the compression, since its atomic theory is what interfaces with the small-cube wavelet coefficients in Section~\ref{sec:hp}.

\subsection{\texorpdfstring{$t_{p}$}{tp}-atoms}

A pair $a = (\beta, \alpha)$ is a \emph{$t_{p}$-atom} if there exists a cube $R \in \mathcal{D}^{0}$ such that:
\begin{enumerate}[(i)]
\item $\scl(\alpha) \subseteq R$ and $\beta_k = 0$ unless $Q_k \subseteq R$
\item \begin{equation}
\label{eq:tp-atom-mass}
  \sum_{Q_k \subseteq R} \abs{\beta_k}^2 \;+\; \sum_{Q \subseteq R} \abs{\alpha(Q)}^2
  \;\le\; \abs{R}^{1 - 2/p} .
\end{equation}
\end{enumerate}
Since the bases lie in $\mathcal{D}^{0}$, a father layer position is present only when $\ell(R) = 1$ and $R = Q_k$, so the father layer of an atom is a single lattice position, and the coordinate pairs $e_k := (\mathbf{1}_{\{k\}}, 0)$ are themselves $t_{p}$-atoms with base $Q_k$, the normalization~\eqref{eq:tp-atom-mass} holding with equality budget $\abs{Q_k}^{1 - 2/p} = 1$.
Any such $R$ is called a \emph{base} of $\alpha$, and we say $\alpha$ is a $t_{p}$-atom \emph{with base $R$} when the base is specified. Since $1 - 2/p < 0$, if \textup{(i)} and \textup{(ii)} hold for $R$ they hold for every $R'' \in \mathcal{D}^{0}$ with $\scl(\alpha) \subseteq R'' \subseteq R$. In particular, the bases of a non-zero $t_{p}$-atom form a finite chain in $\mathcal{D}^{0}$, with a smallest element. The zero sequence is a $t_{p}$-atom with every $R \in \mathcal{D}^{0}$ as base. At $p = 1$, the right-hand side of~\eqref{eq:tp-atom-mass} equals $\abs{R}^{-1}$, recovering~\cite[Definition~4.1.5]{elaoumari2025} on the mother layer.

\begin{proposition}
\label{prop:tp-atom-bound} If $a = (\beta, \alpha)$ is a $t_{p}$-atom with a base $R$, then $a \in t_{p}$ and $\norm{a}_{t_{p}} \le 1$.
\end{proposition}

\begin{proof}
By definition, $S_{\mathrm{loc}}(a)$ is supported in $R$, and $\int_R S_{\mathrm{loc}}(a)^2 = \sum_{Q_k \subseteq R} \abs{\beta_k}^2 + \sum_{Q \subseteq R} \abs{\alpha(Q)}^2 =: \norm{a}_{\ell^2}^2$. By Hölder’s Inequality, we have:
\[
\norm{a}_{t_{p}}^p = \int_{R} S_{\mathrm{loc}}(a)^p \le \abs{R}^{1 - p/2} \, \biggl(\int_{R} S_{\mathrm{loc}}(a)^2 \biggr)^{p/2} = \abs{R}^{1 - p/2} \, \norm{a}_{\ell^2}^p \le \abs{R}^{1 - p/2} \, \biggl( \abs{R}^{1 - 2/p}\biggr)^{p/2} = 1. \qedhere
\]
\end{proof}
Under the compression of Proposition~\ref{prop:Tp-tp-embedding}, the two atomic notions in $T_{p}$ and $t_{p}$ are compatible in the following precise sense: extension by zero maps the mother-layer $t_{p}$-atoms exactly onto the $T_{p}$-atoms whose base lies in $\mathcal{D}^{0}$, and conversely such $T_{p}$-atoms restrict to $t_{p}$-atoms. A $T_{p}$-atom with large base ($\ell(R) > 1$) is carried by $\Theta$ to an element of $t_{p}$ of quasi-norm exactly its own, hence at most $1$, with genuine father mass, in general not a $t_{p}$-atom, its socle being too spread out to fit in a small cube. In the atomic decomposition of Theorem~\ref{thm:tp-atomic-restate}, such contributions are instead resolved into families of small based atoms.

\subsection{The atomic decomposition}

The local dyadic maximal function~\cite[Chapter~5]{meyer1992}, \cite[Definitions~4.1.6--4.1.7]{elaoumari2025} is, for $f \in L^1_{\mathrm{loc}}(\mathbb{R}^n)$:
\[
M_0^d f(x) = \sup_{x \in Q \in \mathcal{D}^{0}} \frac{1}{\abs{Q}}\int_Q \abs{f}.
\]
For $\vartheta \in (0, 1)$ and a measurable set $E$ with $\abs{E} < \infty$, the \emph{local magnification} of $E$ is:
\[
E^\star_0 = \{M^d_0 \mathbf{1}_E > \vartheta\}.
\]
Its properties listed in~\cite[Proposition~4.1.4]{elaoumari2025} (maximality, $\abs{E^\star_0} < \vartheta^{-1}\abs{E}$, $E \subseteq E^\star_0$ up to null sets, monotonicity) are $p$-independent and apply as they stand.

\begin{theorem}
\label{thm:tp-atomic-restate} \textit{(for $p = 1$, on the mother layer, see~\cite[Theorem~4.1.2]{elaoumari2025}.)} Let $0 < p \le 1$ and let $(\beta, \alpha)$ be a pair of scalar sequences, $\beta$ on the unit lattice and $\alpha$ on $\mathcal{D}^{0}$. The following are equivalent:
\begin{enumerate}[(i)]
\item $(\beta, \alpha) \in t_{p}$.
\item There exists $\gamma \in (0,1]$ such that for every $Q \in \mathcal{D}^{0}$ there exists a measurable set $R(Q) \subseteq Q$ with $\abs{R(Q)} \ge \gamma \abs{Q}$, and for every $k \in \mathbb{Z}^n$ a measurable set $R_k \subseteq Q_k$ with $\abs{R_k} \ge \gamma$, such that:
\[
S_{\mathrm{loc}}^R(\beta, \alpha) := \Bigl( \sum_{k \in \mathbb{Z}^n} \abs{\beta_k}^2\, \mathbf{1}_{R_k} \;+\; \sum_{Q \in \mathcal{D}^{0}} \abs{\alpha(Q)}^2\, \abs{Q}^{-1}\, \mathbf{1}_{R(Q)} \Bigr)^{1/2} \in L^p(\mathbb{R}^n).
\]
\item There exist a sequence of $t_{p}$-atoms $\{ a_j \}$ and scalars $\{ \lambda_j \}$ with $\sum_j \abs{\lambda_j}^p < \infty$ and $(\beta, \alpha) = \sum_j \lambda_j a_j$ in $t_{p}$.
\end{enumerate}
Also, the decomposition in \textup{(iii)} may be taken in \emph{restriction form}: there is a partition $\{\Delta_j\}_j$ of $\supp(\beta) \sqcup \supp(\alpha)$ into pairwise disjoint subsets, each contained in the tent of a small cube $R_j \in \mathcal{D}^{0}$, with any father layer position $k$ satisfying $Q_k = R_j$ in its block, such that $\lambda_j a_j = (\beta, \alpha)\, \mathbf{1}_{\Delta_j}$ for every $j$. If $\supp(\beta) \sqcup \supp(\alpha)$ is finite, the decomposition is finite. Finally, $\norm{(\beta, \alpha)}_{t_{p}}^p \simeq \inf \sum_j \abs{\lambda_j}^p$, the infimum over all decompositions as in \textup{(iii)}, with implicit constants depending only on $(n, p, \gamma)$.
\end{theorem}

\begin{proof}
The structure follows~\cite[Chapter~5, proof of Theorem~1, pp.~143--149]{meyer1992}, as adapted in~\cite[Theorem~4.1.2]{elaoumari2025}, with the adjustments required by the quasi-Banach range indicated throughout.

\textbf{(i)$\implies$(ii):} Trivial, with $\gamma = 1$, $R(Q) = Q$, and $R_k = Q_k$: then $S_{\mathrm{loc}}^R(\beta, \alpha) = S_{\mathrm{loc}}(\beta, \alpha) \in L^p$.

\textbf{(iii)$\implies$(i):} If $(\beta, \alpha) = \sum_j \lambda_j a_j$ with $\sum_j \abs{\lambda_j}^p < \infty$ and each $a_j$ a $t_{p}$-atom, then by the $p$-triangle inequality and the uniform atom bound of Proposition~\ref{prop:tp-atom-bound}:
\[
\norm{(\beta, \alpha)}_{t_{p}}^p \;\le\; \sum_j \abs{\lambda_j}^p\, \norm{a_j}_{t_{p}}^p \;\le\; \sum_j \abs{\lambda_j}^p \;<\; \infty ,
\]
both layers of the pair entering through the joint mass condition~\eqref{eq:tp-atom-mass} of the atoms.

\textbf{(ii)$\implies$(iii):} The father layer decomposes exactly: the sub-cubes $R_k$ being pairwise disjoint with $\abs{R_k} \ge \gamma$, its block of $S_{\mathrm{loc}}^R(\beta, \alpha)$ satisfies:
\[
\gamma \sum_{k} \abs{\beta_k}^p \;\le\; \int_{\mathbb{R}^n} \Bigl( \sum_k \abs{\beta_k}^2\, \mathbf{1}_{R_k} \Bigr)^{p/2} \;\le\; \norm{S_{\mathrm{loc}}^R(\beta, \alpha)}_p^p ,
\]
so $(\beta, 0) = \sum_k \beta_k\, e_k$ over the single-site atoms $e_k$, with $\ell^p$ coefficient mass at most $\gamma^{-1} \norm{S_{\mathrm{loc}}^R(\beta, \alpha)}_p^p$, the restriction form holding trivially with the singleton blocks based at the unit cubes. It remains to decompose the mother layer $(0, \alpha)$, whose truncated square function we abbreviate $S_{\mathrm{loc}}^R(\alpha)$, and to concatenate the two families, the $\ell^p$ masses adding by $p$-sub-additivity. Fix $\vartheta \in (0, \gamma)$ and set:
\[
E_k := \{ x \in \mathbb{R}^n : S_{\mathrm{loc}}^R(\alpha)(x) > 2^k \}, \qquad k \in \mathbb{Z} .
\]
Each $E_k$ has finite measure, since $S_{\mathrm{loc}}^R(\alpha) \in L^p$ gives $\abs{E_k} \le 2^{-kp} \norm{S_{\mathrm{loc}}^R(\alpha)}_p^p$ by Chebyshev's Inequality, and $E_{k+1} \subseteq E_k$. We record the $L^p$ layer-cake bound: since $x \in E_k$ exactly when $2^k < S_{\mathrm{loc}}^R(\alpha)(x)$, Tonelli's theorem and the geometric sum $\sum_{k :\, 2^k < t} 2^{kp} \le \frac{2^p}{2^p - 1}\, t^p$ give:
\begin{equation}
\label{eq:layer-cake}
\sum_{k \in \mathbb{Z}} 2^{kp}\, \abs{E_k} \;=\; \int_{\mathbb{R}^n} \sum_{k\,:\; 2^k < S_{\mathrm{loc}}^R(\alpha)(x)} 2^{kp}\, dx \;\le\; \frac{2^p}{2^p - 1}\, \norm{S_{\mathrm{loc}}^R(\alpha)}_p^p .
\end{equation}

For each $k$, consider the local magnification of $E_k$:
\[
E^\star_{k, 0} := \{ M^d_0 \mathbf{1}_{E_k} > \vartheta \} = \bigcup \bigl\{ Q \in \mathcal{D}^{0} : \abs{Q \cap E_k} > \vartheta \abs{Q} \bigr\},
\]
and write $\mathcal{C}^{k} := \{ Q \in \mathcal{D}^{0} : \abs{Q \cap E_k} > \vartheta \abs{Q} \}$ for the family of cubes realizing it. By~\cite[Proposition~4.1.4]{elaoumari2025}, whose properties are $p$-independent and apply as stated: every $Q \in \mathcal{C}^k$ lies below a maximal element of $\mathcal{C}^k$ (maximality is automatic within $\mathcal{D}^{0}$, the dyadic ancestors of side length $\le 1$ forming a finite chain), the maximal cubes $\{Q(k, l)\}_l$ are pairwise disjoint with $E^\star_{k,0} = \bigcup_l Q(k, l)$, and:
\begin{equation}
\label{eq:magnified-measure}
\abs{E^\star_{k, 0}} \;\le\; \vartheta^{-1} \abs{E_k} ,
\end{equation}
by disjointness and $\abs{Q(k,l)} < \vartheta^{-1} \abs{Q(k,l) \cap E_k}$. Since $E_{k+1} \subseteq E_k$ implies $\abs{Q \cap E_{k+1}} \le \abs{Q \cap E_k}$, the families are nested: $\mathcal{C}^{k+1} \subseteq \mathcal{C}^{k}$ (monotonicity of the magnification).

Every supporting cube has a well-defined exit level: let $Q \in \supp(\alpha)$. On $R(Q) \subseteq Q$, of measure $\ge \gamma \abs{Q}$, we have the pointwise lower bound $S_{\mathrm{loc}}^R(\alpha) \ge \abs{\alpha(Q)}\, \abs{Q}^{-1/2} > 0$. Hence, for every $k$ with $2^k < \abs{\alpha(Q)}\, \abs{Q}^{-1/2}$, we have:
\[
\abs{Q \cap E_k} \;\ge\; \abs{R(Q)} \;\ge\; \gamma \abs{Q} \;>\; \vartheta \abs{Q},
\]
so $Q \in \mathcal{C}^{k}$ for all such $k$ (this is the step consuming $\vartheta < \gamma$). Conversely, $Q \in \mathcal{C}^{k}$ forces $\abs{E_k} > \vartheta \abs{Q}$, which fails for all large $k$ since $\abs{E_k} \to 0$ as $k \to +\infty$. Hence, the \emph{exit level}:
\[
k_Q := \max \{ k \in \mathbb{Z} : Q \in \mathcal{C}^{k} \}
\]
is well defined, and $Q \in \mathcal{C}^{k_Q} \setminus \mathcal{C}^{k_Q + 1}$.

For each $k$ and each maximal cube $Q(k, l)$, set:
\[
\Delta(k, l) := \bigl( \mathcal{C}^{k} \setminus \mathcal{C}^{k+1} \bigr) \cap \bigl\{ Q \in \mathcal{D}^{0} : Q \subseteq Q(k, l) \bigr\} .
\]
The sets $\{\Delta(k, l)\}_{k, l}$ are pairwise disjoint: for distinct $k$ because the layers $\mathcal{C}^{k} \setminus \mathcal{C}^{k+1}$ are pairwise disjoint (nestedness), and for equal $k$ and distinct $l$ because the maximal cubes are disjoint. They exhaust $\supp(\alpha)$: given $Q \in \supp(\alpha)$, one has $Q \in \mathcal{C}^{k_Q} \setminus \mathcal{C}^{k_Q + 1}$, and $Q$ lies below some maximal cube $Q(k_Q, l)$, so $Q \in \Delta(k_Q, l)$. Discarding empty blocks, $\{\Delta(k, l)\}$ is a partition of $\supp(\alpha)$ into pairwise disjoint subsets of $\mathcal{D}^{0}$, each consisting of cubes contained in the small cube $Q(k, l) \in \mathcal{D}^{0}$. If $\supp(\alpha)$ is finite, only finitely many blocks are non-empty.

The key estimate: for every non-empty $\Delta(k, l)$:
\begin{equation}
\label{eq:key-estimate}
\sum_{Q \in \Delta(k, l)} \abs{\alpha(Q)}^2 \;\le\; \frac{4^{k+1}}{\gamma - \vartheta}\, \abs{Q(k, l)} .
\end{equation}
This estimate involves no $p$ and transfers unchanged from~\cite[Chapter~5, pp.~146--148]{meyer1992}, \cite[proof of Theorem~4.1.2]{elaoumari2025}: since $Q \notin \mathcal{C}^{k+1}$ gives $\abs{R(Q) \setminus E_{k+1}} \ge (\gamma - \vartheta) \abs{Q}$, integrating $S_{\mathrm{loc}}^R(\alpha)^2 \le 4^{k+1}$ over $Q(k,l) \setminus E_{k+1}$ against the terms of $\Delta(k,l)$ gives it directly.

For each non-empty $\Delta(k, l)$, define:
\[
\lambda(k, l) := \abs{Q(k, l)}^{1/p - 1/2} \Bigl( \sum_{Q \in \Delta(k, l)} \abs{\alpha(Q)}^2 \Bigr)^{1/2}, \qquad \alpha_{k, l} := \lambda(k, l)^{-1}\, \alpha\, \mathbf{1}_{\Delta(k, l)} .
\]
Then: $\scl(\alpha_{k, l}) \subseteq Q(k, l) \in \mathcal{D}^{0}$ and:
\[
\sum_{Q \in \mathcal{D}^{0}} \abs{\alpha_{k, l}(Q)}^2 = \lambda(k, l)^{-2} \sum_{Q \in \Delta(k, l)} \abs{\alpha(Q)}^2 = \abs{Q(k, l)}^{1 - 2/p},
\]
so each $\alpha_{k, l}$ is a $t_{p}$-atom with base contained in $Q(k, l)$, satisfying by construction $\lambda(k, l)\, \alpha_{k, l} = \alpha\, \mathbf{1}_{\Delta(k, l)}$: the decomposition is in restriction form. Since $\{\Delta(k, l)\}$ partitions $\supp(\alpha)$, the equality $\alpha = \sum_{k, l} \lambda(k, l)\, \alpha_{k, l}$ holds in terms of coefficients.

By~\eqref{eq:key-estimate}, with $(4^{k+1})^{p/2} = 2^p\, 2^{kp}$, we have:
\[
\sum_{k, l} \abs{\lambda(k, l)}^p = \sum_{k, l} \abs{Q(k, l)}^{1 - p/2} \Bigl( \sum_{Q \in \Delta(k, l)} \abs{\alpha(Q)}^2 \Bigr)^{p/2} \;\le\; \frac{2^p}{(\gamma - \vartheta)^{p/2}} \sum_k 2^{kp} \sum_l \abs{Q(k, l)} ,
\]
and by~\eqref{eq:magnified-measure} followed by~\eqref{eq:layer-cake}:
\begin{equation}
\label{eq:coeff-bound}
\sum_{k, l} \abs{\lambda(k, l)}^p \;\le\; \frac{2^p\, \vartheta^{-1}}{(\gamma - \vartheta)^{p/2}} \sum_k 2^{kp} \abs{E_k} \;\le\; \frac{4^p\, \vartheta^{-1}} {(2^p - 1)\, (\gamma - \vartheta)^{p/2}}\, \norm{S_{\mathrm{loc}}^R(\alpha)}_p^p \;<\; \infty .
\end{equation}
At $p = 1$ the exponents and constants collapse to those of the thesis.

For the unconditional convergence in $t_{p}$ of the concatenated family: write $\{(\Delta_j, \lambda_j, a_j)\}_j$ for an arbitrary ordering of the non-empty blocks, the father singletons $\{k\}$ with atoms $e_k$ and coefficients $\beta_k$ alongside the mother blocks $\Delta(k, l)$ with atoms $\alpha_{k, l}$, and let $G$ run over finite index sets. The combined $\ell^p$ coefficient mass is finite, by the father bound and~\eqref{eq:coeff-bound}, the masses adding by $p$-sub-additivity, so $a := (\beta, \alpha) \in t_{p}$ by the already proved implication \textup{(iii)}$\implies$\textup{(i)}, i.e.\ $S_{\mathrm{loc}}(\beta, \alpha) \in L^p$. The tail satisfies $a - \sum_{j \in G} \lambda_j\, a_j = a\, \mathbf{1}_{\bigcup_{j \notin G} \Delta_j}$, the restriction of the pair to the omitted blocks, layer by layer (restriction form and partition on each layer), whence, pointwise, we have:
\[
S_{\mathrm{loc}}\Bigl( a\, \mathbf{1}_{\bigcup_{j \notin G} \Delta_j} \Bigr) \;\le\; S_{\mathrm{loc}}(\beta, \alpha), \qquad S_{\mathrm{loc}}\Bigl( a\, \mathbf{1}_{\bigcup_{j \notin G} \Delta_j} \Bigr)(x) \;\longrightarrow\; 0 \quad \text{as } G \uparrow,
\]
the convergence at each $x$ because the sum defining $S_{\mathrm{loc}}(\beta, \alpha)(x)^2$ is a convergent series over the unit cell containing $x$ together with the supporting cubes containing $x$, of which the tail retains only those with index outside $G$. Dominated convergence in $L^p$ gives:
\[
\norm*{ a - \sum_{j \in G} \lambda_j\, a_j
     }_{t_{p}} \;\longrightarrow\; 0 ,
\]
along any exhaustion of the index set, and the tail quasi-norms decrease along inclusion, so the sequential limit is the net limit: the net of finite partial sums converges to $a$, the series being summable to $a = (\beta, \alpha)$ in $t_{p}$, unconditionally.

From \textup{(iii)}$\implies$\textup{(i)}, every decomposition satisfies $\norm{(\beta, \alpha)}_{t_{p}}^p \le \sum_j \abs{\lambda_j}^p$, whence $\norm{(\beta, \alpha)}_{t_{p}}^p \le \inf \sum_j \abs{\lambda_j}^p$. Conversely, applying \textup{(i)}$\implies$\textup{(ii)} (with $\gamma = 1$, $R(Q) = Q$, and $R_k = Q_k$, so that $S_{\mathrm{loc}}^R(\beta, \alpha) = S_{\mathrm{loc}}(\beta, \alpha)$) followed by \textup{(ii)}$\implies$\textup{(iii)} with $\vartheta = 1/2$, the produced concatenated decomposition satisfies $\sum_k \abs{\beta_k}^p + \sum_{k, l} \abs{\lambda(k, l)}^p \le C'(n, p)\, \norm{S_{\mathrm{loc}}(\beta, \alpha)}_p^p$, for a constant $C'(n, p)$, by the father bound and~\eqref{eq:coeff-bound}, so $\inf \sum_j \abs{\lambda_j}^p \le C'(n, p)\, \norm{(\beta, \alpha)}_{t_{p}}^p$. \end{proof}

\section{Wavelet characterization of \texorpdfstring{$h^p(\mathbb{R}^n)$}{hp(Rn)}}
\label{sec:hp}

\label{ssec:hp-lemmas}

For $f \in \mathcal{S}'(\mathbb{R}^n)$ with wavelet coefficients $\beta_k$, $\alpha(\lambda)$ as in Subsection~\ref{ssec:wavelet}, the local square functions are:
\begin{equation}
\label{eq:SLoc-def}
\begin{aligned}
S_{\mathrm{Loc}}(f) &:= \biggl(\sum_{k \in \mathbb{Z}^n} \abs{\beta_k}^2\, \abs{\phi_k}^2 + \sum_{\lambda \in \Lambda^0} \abs{\alpha(\lambda)}^2\, \abs{\psi_\lambda}^2\biggr)^{1/2}, \\
S_{\mathrm{Loc}}^R(f) &:= \biggl(\sum_{k \in \mathbb{Z}^n} \abs{\beta_k}^2\, \mathbf{1}_{R_k} + \sum_{\lambda \in \Lambda^0} \abs{\alpha(\lambda)}^2\, \abs{Q(\lambda)}^{-1}\, \mathbf{1}_{R(\lambda)}\biggr)^{1/2},
\end{aligned}
\end{equation}
with $R(\lambda) \subseteq Q(\lambda)$, $R_k \subseteq Q_k$ the sub-cubes of Subsection~\ref{ssec:wavelet}, of measure ratio $\gamma$, and $S_{\mathrm{Loc}}^Q(f)$ the same with $Q_k$, $Q(\lambda)$ in place of $R_k$, $R(\lambda)$.

\begin{lemma}
\label{lem:khinchin-p} Let $0 < p \le 1$. Write $u_k := \phi_k$ for $k \in \mathbb{Z}^n$ and $u_\lambda := \psi_\lambda$ for $\lambda \in \Lambda^0$. There exists $C(p) < \infty$ such that for every finite subset $F \subseteq \mathbb{Z}^n \sqcup \Lambda^0$, every $x \in \mathbb{R}^n$, and every sequence $\{c(\mu)\}_{\mu \in F}$ of scalars:
\[
\biggl( \sum_{\mu \in F} \abs{c(\mu)}^2\, \abs{u_\mu(x)}^2 \biggr)^{1/2} \;\le\; C(p)\, \biggl( \frac{1}{2^{\abs{F}}} \sum_{\varepsilon \in \{\pm 1\}^F} \Bigl| \sum_{\mu \in F} \varepsilon(\mu)\, c(\mu)\, u_\mu(x) \Bigr|^{p} \biggr)^{1/p}.
\]
For $p = 1$ this is the classical Khintchine inequality with $C(1) = \sqrt{2}$.
\end{lemma}

\begin{proof}
Fix $x \in \mathbb{R}^n$ and apply the scalar Khintchine inequality (lower half) to the finite sequence $a_\mu := c(\mu)\, u_\mu(x)$, $\mu \in F$: for every $0 < p < \infty$ there is $A_p > 0$ with:
\[
A_p \Bigl( \sum \abs{a_\lambda}^2 \Bigr)^{1/2} \;\le\;
\Bigl( \mathbb{E}_\varepsilon \abs[\Big]{\sum \varepsilon(\lambda)\, a_\lambda}^p \Bigr)^{1/p},
\]
where $\mathbb{E}_\varepsilon$ denotes the average over $\varepsilon \in \{\pm 1\}^F$. Take $C(p) = A_p^{-1}$. For the inequality in the range $0 < p < \infty$, including $0 < p < 1$, see~\cite[Appendix~C]{Grafakos2014}. The optimal constants are determined in~\cite{haagerup1981}.
\end{proof}

\begin{lemma}
\label{lem:sf-comp} Let $f \in \mathcal{S}'(\mathbb{R}^n)$ have wavelet coefficients $\beta_k = \pair{f}{\phi_k}$ and $\alpha(\lambda) = \pair{f}{\psi_\lambda}$, in the sense of the pairing conventions of Subsection~\ref{ssec:wavelet}, and write $S_{\mathrm{Loc}}^{m Q}(f)$ for $S_{\mathrm{Loc}}^Q(f)$ of~\eqref{eq:SLoc-def} with $m Q_k$, $m Q(\lambda)$ in place of $Q_k$, $Q(\lambda)$. Then, pointwise on $\mathbb{R}^n$:
\[
  c\, S_{\mathrm{Loc}}^{R}(f)(x)
  \;\le\; S_{\mathrm{Loc}}(f)(x)
  \;\le\; C_\psi\, S_{\mathrm{Loc}}^{m Q}(f)(x),
\]
where $c > 0$ is the common lower constant of~\eqref{eq:lb-psi} and of the father clause of Subsection~\ref{ssec:wavelet}, and $C_\psi := \max\bigl( \sup_e \norm{\psi^e}_\infty, \norm{\phi}_{L^\infty(\mathbb{R})}^n \bigr)$, both depending only on the basis. The bounds hold term-wise on each layer, as the proof shows.
\end{lemma}

\begin{proof}
Both inequalities follow term-wise from:
\[
c\, \abs{Q(\lambda)}^{-1/2}\, \mathbf{1}_{R(\lambda)}(x) \;\le\; \abs{\psi_\lambda(x)} \;\le\; C_\psi\, \abs{Q(\lambda)}^{-1/2}\, \mathbf{1}_{m Q(\lambda)}(x),
\]
and its father companion $c\, \mathbf{1}_{R_k}(x) \le \abs{\phi_k(x)} \le \norm{\phi}_{L^\infty(\mathbb{R})}^n\, \mathbf{1}_{m Q_k}(x)$, the left inequalities being~\eqref{eq:lb-psi} and the father clause of Subsection~\ref{ssec:wavelet}, and the right ones following from $\supp(\phi_k) \subseteq m Q_k$ and the identity $\norm{\psi_\lambda}_\infty = 2^{nj/2} \norm{\psi^e}_\infty = \abs{Q(\lambda)}^{-1/2} \norm{\psi^e}_\infty$ together with $\supp(\psi_\lambda) \subseteq m Q(\lambda)$: multiply by the coefficient moduli, square, sum over both layers, and take square roots (all terms are non-negative, both sides possibly infinite).
\end{proof}

\begin{lemma}
\label{lem:tp-to-sf} Let $a = (\beta, \alpha)$ be a pair of scalar sequences, $\beta$ on the unit lattice and $\alpha$ on $\Lambda^0$. For $e \in E$, define $\alpha^e = \{\alpha^e(Q)\}_{Q \in \mathcal{D}^{0}}$ by $\alpha^e(Q) := \alpha(\lambda)$ for the unique $\lambda = k/2^j + e/2^{j+1} \in \Lambda^0$ of direction $e$ with $Q(\lambda) = Q$, and write $S_{\mathrm{loc}}(\beta, \alpha)$ for the local square function~\eqref{eq:Sz-p} of the pair, the mother sum running over $\Lambda^0$ through the cubes $Q(\lambda)$. Then:
\[
\max\Bigl( \norm{(\beta, 0)}_{t_{p}}^p ,\ \max_{e \in E}\, \norm{(0, \alpha^e)}_{t_{p}}^p \Bigr) \;\le\; \norm{S_{\mathrm{loc}}(\beta, \alpha)}_p^p \;\le\; \norm{(\beta, 0)}_{t_{p}}^p + \sum_{e \in E} \norm{(0, \alpha^e)}_{t_{p}}^p .
\]
In particular, $S_{\mathrm{loc}}(\beta, \alpha)$ lies in $L^p(\mathbb{R}^n)$ if and only if the father pair $(\beta, 0)$ and every pair $(0, \alpha^e)$, $e \in E$, lie in $t_{p}$. This applies, in particular, to the wavelet coefficient sequences $\beta_k = \pair{f}{\phi_k}$ and $\alpha(\lambda) = \pair{f}{\psi_\lambda}$ of any $f \in \mathcal{S}'(\mathbb{R}^n)$ for which the coefficients are defined, in which case $S_{\mathrm{loc}}(\beta, \alpha) = S_{\mathrm{Loc}}^Q(f)$, the square function of condition (D).
\end{lemma}

\begin{proof}
For fixed $e$, the map $\lambda \mapsto Q(\lambda)$ is a bijection between the indices of direction $e$ in $\Lambda^0$ and $\mathcal{D}^{0}$, so splitting the sum over $\lambda$ according to direction factorizes $S_{\mathrm{loc}}(\beta, \alpha)$ pointwise:
\begin{equation}
\label{eq:D-factorization}
S_{\mathrm{loc}}(\beta, \alpha) = \biggl( \sum_{k \in \mathbb{Z}^n} \abs{\beta_k}^2\, \mathbf{1}_{Q_k} \;+ \sum_{\lambda \in \Lambda^0} \abs{\alpha(\lambda)}^2\, \abs{Q(\lambda)}^{-1}\, \mathbf{1}_{Q(\lambda)} \biggr)^{1/2} = \Bigl( S_{\mathrm{loc}}(\beta)^2 + \sum_{e \in E} S_{\mathrm{loc}}(\alpha^e)^2 \Bigr)^{1/2},
\end{equation}
the inner bracket for each $e$ being $S_{\mathrm{loc}}(\alpha^e)^2$ by~\eqref{eq:Sz-p}, and the father term being $S_{\mathrm{loc}}(\beta)^2$ exactly, the unit cubes tiling $\mathbb{R}^n$. For the norm bounds: the lower one follows from~\eqref{eq:D-factorization} by sub-additivity of $t \mapsto t^{p/2}$ and monotonicity, discarding all but one layer or one direction and integrating, while the upper one follows from $\bigl( \sum g \bigr)^{p/2} \le \sum g^{p/2}$ with the $g$'s running over $S_{\mathrm{loc}}(\beta)^2$ and $S_{\mathrm{loc}}(\alpha^e)^2$, $e \in E$ (sub-additivity of $t \mapsto t^{p/2}$, as $p/2 \le 1$), integrated with respect to $x$.
\end{proof}

\begin{lemma}
\label{lem:dist-conv} Let $f \in \mathcal{S}'(\mathbb{R}^n)$ be such that the inhomogeneous wavelet decomposition~\eqref{eq:inh-wavelet-decomp} converges to $f$ in $\mathcal{S}'(\mathbb{R}^n)$ (e.g.\ $f$ of order less than $r$, convention \textup{(P1)}), with coefficients $\beta(k)$, $\alpha(\lambda)$ satisfying $\sum_k \abs{\beta(k)}^p < \infty$. Suppose each $\{\alpha^e(Q)\}_{Q \in \mathcal{D}^{0}}$ admits a $t_{p}$-atomic decomposition $\alpha^e = \sum_j \lambda^e_j\, \alpha^e_j$ with $\sum_j \abs{\lambda^e_j}^p < \infty$ and bases $Q^e_j \in \mathcal{D}^{0}$. Then for every $\varphi \in \mathcal S(\mathbb{R}^n)$:
\[
\pair{f}{\varphi} = \sum_{k \in \mathbb{Z}^n} \beta_k\, \pair{\phi_k}{\varphi}
+ \sum_{e \in E} \sum_{j = 1}^{\infty} \lambda^e_j \Bigl( \sum_{Q \in \mathcal{D}^{0}} \alpha^e_j(Q)\, \pair{\psi^e_Q}{\varphi} \Bigr),
\]
with absolute convergence of all sums.
\end{lemma}

\begin{proof}
By assumption, $\pair{f}{\varphi} = \sum_k \beta_k \pair{\phi_k}{\varphi} + \sum_{e} \sum_{Q \in \mathcal{D}^{0}} \alpha^e(Q) \pair{\psi^e_Q}{\varphi}$, so it suffices to prove absolute convergence of the displayed sums. The regrouping of the generated wavelet part by the blocks of the atomic decomposition is then legitimate, using $\lambda^e_j \alpha^e_j = \alpha^e \mathbf{1}_{\Delta^e_j}$ in restriction form (Theorem~\ref{thm:tp-atomic-restate}).

Since $\sum_k \abs{\beta_k}^p < \infty$ gives $\sup_k \abs{\beta_k} < \infty$, and $\varphi \in \mathcal S$ pairs with the uniformly translated family $\{\phi_k\}$ with rapid decay, $\abs{\pair{\phi_k}{\varphi}} \le C_N (1 + \abs{k})^{-N}$ for every $N$, the father wavelet sum converges absolutely.

For the generated wavelet part, we show, for each $e$ and $j$:
\begin{equation}
\label{eq:dist-conv-inner}
\sum_{Q \in \mathcal{D}^{0}} \abs{\alpha^e_j(Q)}\, \abs{\pair{\psi^e_Q}{\varphi}} \;\le\; C^{1/2},
\end{equation}
with $C$ depending only on $(n, p, \varphi, \text{wavelet basis})$. The outer sum then converges absolutely since $\sum_j \abs{\lambda^e_j} \le \bigl( \sum_j \abs{\lambda^e_j}^p \bigr)^{1/p} < \infty$ and $\# E < \infty$. Since $\scl(\alpha^e_j) \subseteq Q^e_j$, Cauchy-Schwarz Inequality gives:
\[
\sum_{Q \in \mathcal{D}^{0}} \abs{\alpha^e_j(Q)}\, \abs{\pair{\psi^e_Q}{\varphi}}
\le \Bigl( \sum_{Q \subseteq Q^e_j} \abs{\alpha^e_j(Q)}^2 \Bigr)^{1/2}
\Bigl( \sum_{Q \subseteq Q^e_j} \abs{\pair{\psi^e_Q}{\varphi}}^2 \Bigr)^{1/2}.
\]
The $t_{p}$-atom mass condition bounds the first factor by $\abs{Q^e_j}^{(1 - 2/p)/2}$. For the second, we use the Carleson-type condition for $\varphi \in \mathcal S(\mathbb{R}^n) \subseteq \dot\Lambda_{n_p}(\mathbb{R}^n)$ (for $p = 1$, $\mathcal S \subseteq \mathrm{BMO}$): for a uniform $C = C(n, p, \varphi, \text{wavelet basis})$,
\begin{equation}
\label{eq:carleson}
\sum_{Q \subseteq Q^e_j} \abs{\pair{\psi^e_Q}{\varphi}}^2 \;\le\; C\, \abs{Q^e_j}^{2/p - 1}.
\end{equation}
At $p = 1$, \eqref{eq:carleson} is the classical Carleson condition characterizing $\mathrm{BMO}$ by wavelet coefficients~\cite[Chapter~5, Theorem~4, pp.~154--155]{meyer1992}, and for $p < 1$ it follows from the pointwise coefficient bound of $\dot\Lambda_{n_p}$, $\abs{\pair{\psi^e_Q}{\varphi}} \le C' \abs{Q}^{1/p - 1/2}$ (see~\cite{frazier-jawerth1990}), by summing over the dyadic sub-cubes of $Q^e_j$: at scale $2^{-i} \le \ell(Q^e_j)$ there are $\abs{Q^e_j}\, 2^{in}$ such cubes, each contributing $\abs{Q}^{2/p - 1} = 2^{-in(2/p - 1)}$, and the geometric series $\abs{Q^e_j} \sum_i 2^{-in(2/p - 2)}$ converges (as $2/p - 2 > 0$) and is dominated by its top scale, $\simeq \abs{Q^e_j}^{2/p - 1}$. Multiplying the two factors, the powers of $\abs{Q^e_j}$ cancel exactly: $\abs{Q^e_j}^{(1 - 2/p)/2}\, \abs{Q^e_j}^{(2/p - 1)/2} = 1$, which proves~\eqref{eq:dist-conv-inner}.
\end{proof}

\begin{lemma}
\label{lem:atom-builder} Fix $e \in E$.
\begin{enumerate}[label=\textup{(\roman*)}, wide, leftmargin=0pt]
\item Let $(\beta, \alpha)$ be a $t_{p}$-atom with base $R \in \mathcal{D}^{0}$. Define:
\[
a(x) := \sum_{k \in \mathbb{Z}^n} \beta_k\, \phi_k(x) \;+ \sum_{Q \in \mathcal{D}^{0}} \alpha(Q)\, \psi^e_Q(x).
\]
Then the series converges unconditionally in $L^2(\mathbb{R}^n)$, $\supp(a) \subseteq m R$, and $\kappa_0^{-1} a$ is an $L^2$ $h^{p}$-atom associated with the cube $m R$, where $\kappa_0 = m^{n(1/p - 1/2)}$. The atom is small-cube when $\ell(m R) = m\, \ell(R) \leq 1$ and large-cube otherwise. In the small case the father layer vanishes, its clause forcing $\ell(R) = 1$ whenever $\beta \neq 0$, and all moments of $a$ up to order $N_p$ vanish, while in the large case no moment conditions are required.
\item Let $(\beta, \alpha^e) \in t_{p}$ and let $(\beta, \alpha^e) = \sum_j \lambda_j b_j$ be a decomposition in restriction form as in Theorem~\ref{thm:tp-atomic-restate}: pairwise disjoint blocks $\Delta_j \subseteq \mathbb{Z}^n \sqcup \mathcal{D}^{0}$ partitioning $\supp(\beta) \sqcup \supp(\alpha^e)$, the father blocks singletons $\{k\}$ based at $Q_k$, bases $R_j \in \mathcal{D}^{0}$, and $\lambda_j b_j = (\beta, \alpha^e)\, \mathbf{1}_{\Delta_j}$, layer by layer. Let $a_j$ be the function built from $b_j$ by \textup{(i)}, so that $\kappa_0^{-1} a_j$ is an $L^2$ $h^{p}$-atom associated with the cube $m R_j$, the father singletons contributing $\lambda_j\, a_j = \beta_k\, \phi_k$. Then each block series can be written as:
\[
\lambda_j\, a_j \;=\; \sum_{k \in \Delta_j} \beta_k\, \phi_k \;+ \sum_{Q \in \Delta_j} \alpha^e(Q)\,
\psi^e_Q ,
\]
one of the two sums empty for each block,
and converges unconditionally in $L^2(\mathbb{R}^n)$:
\[
\sum_j \abs{\kappa_0 \lambda_j}^p \;\le\; \kappa_0^p\,
C(n, p, \gamma)\, \norm{(\beta, \alpha^e)}_{t_{p}}^p ,
\]
and the series $\displaystyle \sum_j \lambda_j a_j = \sum_j (\kappa_0 \lambda_j) \bigl( \kappa_0^{-1} a_j \bigr)$ converges unconditionally in $h^p(\mathbb{R}^n)$, and its sum satisfies:
\[
\sum_j \lambda_j\, a_j \;=\; \sum_{k \in \mathbb{Z}^n} \beta_k\, \phi_k \;+ \sum_{Q \in \mathcal{D}^{0}} \alpha^e(Q)\,
\psi^e_Q \qquad \text{in } \mathcal{S}'(\mathbb{R}^n),
\]
the right-hand side converging unconditionally in $\mathcal{S}'(\mathbb{R}^n)$ (Lemma~\ref{lem:dist-conv}). If $\supp(\beta) \sqcup \supp(\alpha^e)$ is finite, the decomposition is finite and each $a_j$ is a finite linear combination of $\{\phi_k\}_{k \in \Delta_j} \cup \{\psi^e_Q\}_{Q \in \Delta_j}$.
\end{enumerate}
\end{lemma}

\begin{proof}
For part \textup{(i)}: the family $\{\phi_k\}_{k \in \mathbb{Z}^n} \cup \{\psi^e_Q\}_{Q \in \mathcal{D}^{0}}$ is orthonormal in $L^2(\mathbb{R}^n)$, being part of the inhomogeneous basis, and the joint $t_{p}$-atom condition~\eqref{eq:tp-atom-mass} makes the pair square-summable. Hence, the series converges unconditionally in $L^2$ with:
\[
\norm{a}_2^2 = \sum_{k \in \mathbb{Z}^n} \abs{\beta_k}^2 + \sum_{Q \in \mathcal{D}^{0}} \abs{\alpha(Q)}^2 \le
\abs{R}^{1 - 2/p},
\]
by Parseval.

For the support, every $Q \in \supp(\alpha)$ satisfies $Q \subseteq R$ (the socle condition $\scl(\alpha) \subseteq R$), hence $\supp(\psi^e_Q) \subseteq m Q \subseteq m R$, and every $k$ with $\beta_k \neq 0$ has $Q_k \subseteq R$ (the father clause of the atom), hence $\supp(\phi_k) \subseteq m Q_k \subseteq m R$. Each partial sum is supported in $m R$, and $L^2$-limits preserve essential support, so $\supp(a) \subseteq m R$.

For the moments, we may assume the small-cube case $m\, \ell(R) \le 1$, where they are required: there $\ell(R) < 1$, and the father layer vanishes, since a unit cube $Q_k \subseteq R$ would force $\ell(R) \ge 1$. Each $\psi^e_Q$ has vanishing moments up to order $N \ge N_p + 1$ by~\eqref{eq:moments}. For $\abs{\nu} \le N_p$, the partial sums $a_F := \sum_{Q \in F} \alpha(Q) \psi^e_Q$, $F$ finite, converge to $a$ in $L^2(m R)$, hence in $L^1(m R)$ by Hölder's inequality on the bounded set $m R$. Since $x^\nu\, \mathbf{1}_{m R}$ is bounded, integration against it is continuous on $L^1(m R)$, so:
\[
\int_{\mathbb{R}^n} x^\nu\, a(x)\, dx = \lim_{F} \sum_{Q \in F}
\alpha(Q) \int_{\mathbb{R}^n} x^\nu\, \psi^e_Q(x)\, dx = 0 .
\]

The cube $m R$ has measure $\abs{m R} = m^n \abs{R}$. Since $1/2 - 1/p \le 0$:
\[
\norm{a}_2 \le \abs{R}^{1/2 - 1/p} = \bigl( m^n \bigr)^{1/p - 1/2}\, \abs{m R}^{1/2 - 1/p} = \kappa_0\, \abs{m R}^{1/2 - 1/p},
\]
so $\kappa_0^{-1} a$ satisfies the support and size conditions of an $L^2$ $h^{p}$-atom on the cube $m R$, the norm computed over both layers. When $\ell(m R) \leq 1$ it also satisfies the required moment conditions by the previous step, and when $\ell(m R) > 1$ no moment conditions are required.

For part \textup{(ii)}: each $b_j$ is a $t_{p}$-atom, so \textup{(i)} applies to it with base $R_j$: the built series $a_j$ converges unconditionally in $L^2(\mathbb{R}^n)$, the father singleton blocks giving $a_j = \phi_k$ outright. Multiplying by $\lambda_j$ and using the restriction identity $\lambda_j b_j = (\beta, \alpha^e)\, \mathbf{1}_{\Delta_j}$ term-wise on each layer:
\[
\lambda_j\, a_j \;=\; \sum_{k \in \Delta_j} \beta_k\, \phi_k \;+ \sum_{Q \in \Delta_j} \alpha^e(Q)\, \psi^e_Q ,
\]
discarding the vanishing terms, which affects neither the sum nor the unconditionality of its $L^2$ convergence. Only the blocks, being atoms, are square-summable.

By Theorem~\ref{thm:tp-atomic-restate}, the decomposition satisfies $\sum_j \abs{\lambda_j}^p \le C(n, p, \gamma)\, \norm{(\beta, \alpha^e)}_{t_{p}}^p$ (the father bound and~\eqref{eq:coeff-bound}), so multiplying by $\kappa_0^p$ gives the claim.

For the unconditional convergence in $h^{p}$: each $\kappa_0^{-1} a_j$ is an $L^2$ $h^{p}$-atom, so $\norm{\kappa_0^{-1} a_j}_{h^{p}} \le C_{\mathrm{at}}$ uniformly in $j$, by the uniform atomic bound~\eqref{eq:atom-hp-bound}. By $p$-sub-additivity of the quasi-norm, for every finite index set $G$:
\[
\norm[\Big]{\sum_{j \in G} \lambda_j\, a_j}_{h^{p}}^p \;\le\; C(n, p)^p \sum_{j \in G} \abs{\kappa_0 \lambda_j}^p ,
\]
so the $\ell^p$-summability of $\{\kappa_0 \lambda_j\}$ verifies the criterion of Lemma~\ref{lem:net-summability}, and the series is summable in $h^{p}$: the net of finite partial sums converges, unconditionally. Denote its sum by $f^e$. Since $h^{p}$ embeds continuously in $\mathcal{S}'(\mathbb{R}^n)$, the series also converges to $f^e$ in $\mathcal{S}'(\mathbb{R}^n)$.

For the identity, fix $\varphi \in \mathcal{S}(\mathbb{R}^n)$: applying Lemma~\ref{lem:dist-conv} to the mother layer of the decomposition at hand, and using the restriction identity together with Tonelli's Theorem:
\[
\sum_{Q \in \mathcal{D}^{0}} \abs{\alpha^e(Q)}\, \abs{\pair{\psi^e_Q}{\varphi}}
\;=\; \sum_j \abs{\lambda_j} \sum_{Q \in \Delta_j} \abs{\alpha_j(Q)}\, \abs{\pair{\psi^e_Q}{\varphi}} \;\le\; C^{1/2} \sum_j \abs{\lambda_j} \;<\; \infty ,
\]
by the uniform block bound~\eqref{eq:dist-conv-inner} and $\ell^p \hookrightarrow \ell^1$, while the father layer contributes $\sum_k \abs{\beta_k}\, \abs{\pair{\phi_k}{\varphi}} \le \bigl( \sup_k \abs{\pair{\phi_k}{\varphi}} \bigr) \sum_k \abs{\beta_k} < \infty$, again by $\ell^p \hookrightarrow \ell^1$. In particular, the combined synthesis $\sum_k \beta_k \phi_k + \sum_Q \alpha^e(Q) \psi^e_Q$ defines a tempered distribution $g$, with $\pair{g}{\varphi} = \sum_k \beta_k \pair{\phi_k}{\varphi} + \sum_Q \alpha^e(Q) \pair{\psi^e_Q}{\varphi}$ independently of the order of summation, and the convergence to $g$ is unconditional in $\mathcal{S}'(\mathbb{R}^n)$. On the other hand, $\mathcal{S}'(\mathbb{R}^n)$-convergence of the atomic series gives $\pair{f^e}{\varphi} = \sum_j \lambda_j \pair{a_j}{\varphi}$, and pairing the $L^2$-convergent block series against $\varphi \in \mathcal{S} \subseteq L^2$ gives:
\[
\lambda_j \pair{a_j}{\varphi} = \sum_{k \in \Delta_j} \beta_k\, \pair{\phi_k}{\varphi} \;+ \sum_{Q \in \Delta_j} \alpha^e(Q)\, \pair{\psi^e_Q}{\varphi} .
\]
Since $\{\Delta_j\}$ partitions $\supp(\beta) \sqcup \supp(\alpha^e)$ and the full scalar family is absolutely summable, the iterated sum may be regrouped freely:
\[
\pair{f^e}{\varphi} = \sum_j \biggl( \sum_{k \in \Delta_j} \beta_k\, \pair{\phi_k}{\varphi} \;+ \sum_{Q \in \Delta_j} \alpha^e(Q)\, \pair{\psi^e_Q}{\varphi} \biggr) = \pair{g}{\varphi} .
\]
As $\varphi \in \mathcal{S}$ was arbitrary, $f^e = g$ in $\mathcal{S}'(\mathbb{R}^n)$.

For the finite case, Theorem~\ref{thm:tp-atomic-restate} gives a finitely supported pair admits a finite decomposition, and each $\Delta_j$ is then finite, so each $\lambda_j a_j$ is a finite linear combination.
\end{proof}

\begin{lemma}
\label{lem:schwartz-correction} Let $0 < p \le 1 < q \le \infty$ and let $\varphi \in \mathcal S(\mathbb{R}^n)$ satisfy $\int_{\mathbb{R}^n} \varphi = 1$ and $\int_{\mathbb{R}^n} x^\nu\, \varphi(x)\, dx = 0$ for $0 < \abs{\nu} \le N_p$. Let $A$ be a large-ball $L^q$ $h^{p}$-atom on $B_A = B(x_A, r_A)$, $r_A \ge 1$, and $a$ a small-ball $L^q$ $h^{p}$-atom on $B_a = B(x_a, r_a)$, $r_a < 1$. Then, with vanishing moments up to order $N_p$ in every case, and for every $\varepsilon > 0$:
\begin{enumerate}[label=\textup{(\roman*)}, wide, leftmargin=0pt]
\item $A_\varphi := A - \varphi * A$ is, up to a multiplicative constant depending only on $(n, p, q, \varepsilon, \varphi)$, a classical $(p, q, \varepsilon)$-molecule of $H^p(\mathbb{R}^n)$ associated with $B_A$.
\item $a_\varphi := \varphi * a$ is, up to a multiplicative constant depending only on $(n, p, \varepsilon, \varphi)$, a classical $(p, \tilde q, \varepsilon)$-molecule of $H^p(\mathbb{R}^n)$ associated with the translated unit ball $B(x_a, 1)$, at once for every $1 < \tilde q \le \infty$.
\item There is $C = C(n, p, \varphi)$, independent of the atom and of its ball, with $\norm{\varphi * A}_p \le C$ and $\norm{\varphi * a}_p \le C$.
\end{enumerate}
\end{lemma}

\begin{proof}
Fix $\varepsilon > 0$ and set $M := n/p + \varepsilon$. Since $\varphi \in \mathcal S(\mathbb{R}^n)$, there is, for every $N \ge 0$, a constant $C_N(\varphi)$ with $\abs{\partial^\mu \varphi(x)} \le C_N(\varphi)\, (1 + \abs{x})^{-N}$ for every multi-index $\mu$ with $\abs{\mu} \le N_p' + 1$ and every $x \in \mathbb{R}^n$; fix $N := M$ and write $C_\varphi := C_M(\varphi)$. In particular $\varphi \in L^1$ with $\norm{\varphi}_1 \le C_\varphi \int_{\mathbb{R}^n} (1 + \abs{y})^{-M}\, dy < \infty$, all moment integrals below converge absolutely, and Fubini's Theorem applies throughout. The separation estimate: for $z \in B(x_0, r_0)$ and $\abs{x - x_0} \ge 2 r_0$, $\abs{x - z} \ge \tfrac12 \abs{x - x_0}$ gives $(1 + \abs{x - z})^{-M} \le 2^M (1 + \abs{x - x_0})^{-M}$, and when $r_0 \le 1$ the same bound holds for all $x$ up to a constant, since on $\abs{x - x_0} \le 2 r_0 \le 2$ one has $(1 + \abs{x-z})^{-M} \le 1$ while $(1+\abs{x-x_0})^{-M} \ge 3^{-M}$. Finally, every $L^q$ $h^{p}$-atom satisfies $\norm{a}_1 \le \abs{B}^{1-1/p}$ by Hölder's Inequality on its ball $B$.

For part (i), the moments. By Fubini's Theorem, the substitution $z = x - y$, and the Binomial Theorem, for $\abs{\nu} \le N_p$:
\[
  \int (\varphi * A)(x)\, x^\nu\, dx
  = \sum_{\beta \le \nu} \binom{\nu}{\beta}
     \Bigl( \int A(z)\, z^\beta\, dz \Bigr)
     \Bigl( \int \varphi(y)\, y^{\nu - \beta}\, dy \Bigr) .
\]
The term $\beta = \nu$ equals $\int A\, x^\nu\, dx$ since $\int \varphi = 1$, and every term $\beta < \nu$ vanishes since $\int \varphi(y)\, y^{\nu-\beta}\, dy = 0$ for $0 < \abs{\nu - \beta} \le N_p$. Hence $\int A_\varphi\, x^\nu\, dx = 0$ for all $\abs{\nu} \le N_p$.

For the size, by the Triangle and Young Inequalities:
\begin{equation}
\label{eq:mol-size-A-schwartz}
  \norm{A_\varphi}_q
  \;\le\; \bigl( 1 + \norm{\varphi}_1 \bigr)
     \norm{A}_q
  \;\le\; \bigl( 1 + \norm{\varphi}_1 \bigr)
     \abs{B_A}^{1/q - 1/p} .
\end{equation}
As for the tail, for $\abs{x - x_A} \ge 2 r_A$ we have $A_\varphi(x) = -(\varphi * A)(x)$, and by the separation estimate together with $\norm{A}_1 \le \abs{B_A}^{1-1/p}$:
\[
  \abs{A_\varphi(x)}
  \;\le\; \norm{A}_1
     \sup_{y \in B_A} \abs{\varphi(x - y)}
  \;\le\; C_\varphi\, 2^M\, \abs{B_A}^{1 - 1/p}\,
     (1 + \abs{x - x_A})^{-M} .
\]
Restricting to $A_k(B_A)$, $k \ge 1$, of inner radius $s_k = 2^k r_A \ge 2^k$, this gives $\abs{A_\varphi(x)} \le C\, \abs{B_A}^{1-1/p}\, s_k^{-M}$ there, and since $\abs{B_A}^{1-1/p} = \omega_n^{1-1/p}\, r_A^{-n_p} \le \omega_n^{1-1/p}$ as $r_A \ge 1$, $\abs{A_k(B_A)}^{1/q} = \bigl( (2^n-1)\, \omega_n \bigr)^{1/q}\, s_k^{\,n/q}$, and $M = n/p + \varepsilon$:
\[
  \norm{A_\varphi}_{L^q(A_k(B_A))}
  \;\le\; C(n, q, \varphi)\; s_k^{\,n/q - n/p - \varepsilon}
  \;\le\; C(n,q,\varphi)\, 2^{-k\varepsilon}\, \abs{2^k B_A}^{1/q-1/p} ,
\]
using $s_k = 2^k r_A \ge 2^k$, so that $s_k^{-\varepsilon} \le 2^{-k\varepsilon}$ as $r_A \ge 1$, and $\abs{2^kB_A}^{1/q-1/p} = \omega_n^{1/q-1/p}\, s_k^{\,n/q-n/p}$. Together with~\eqref{eq:mol-size-A-schwartz} on $2B_A$, this is condition (i) of a classical $(p,q,\varepsilon)$-molecule associated with $B_A$, up to a constant depending only on $(n,p,q,\varepsilon,\varphi)$, and the moments were verified above.

For part (ii), the moments, as before. The same binomial expansion gives $\int a_\varphi(x)\, x^\nu\, dx = 0$ for $\abs{\nu} \le N_p$, every factor $\int a(z)\, z^\beta\, dz$ vanishing by the moment conditions of the small atom, so no condition on $\varphi$ is used here.

For the pointwise envelope, for $p = 1$ this is immediate: $\abs{a_\varphi(x)} \le \norm{a}_1 \sup_{y \in B_a} \abs{\varphi(x-y)} \le C_\varphi\, 2^M\, (1+\abs{x-x_a})^{-M}$, using $\norm{a}_1 \le 1$ and the separation estimate. For $p < 1$, let $P_x$ be the Taylor polynomial of degree $N_p'$ of $y \mapsto \varphi(x-y)$ at $y = x_a$. The moments of $a$, of order up to $N_p \ge N_p'$, annihilate it, so $a_\varphi(x) = \int_{B_a} [\varphi(x-y) - P_x(y)]\, a(y)\, dy$, and Taylor's Theorem with the derivative bound $\abs{\partial^\mu \varphi(x-z)} \le C_\varphi (1+\abs{x-z})^{-M}$ for $\abs{\mu} = N_p' + 1$, $z \in [x_a, y] \subseteq B_a$, together with the separation estimate at $r_0 = r_a \le 1$, gives:
\[
  \abs{\varphi(x-y) - P_x(y)}
  \;\le\; C(n, N_p')\, C_\varphi\, (1+\abs{x-x_a})^{-M}\, r_a^{\,N_p'+1} .
\]
Multiplying by $\norm{a}_1 \le \abs{B_a}^{1-1/p} = \omega_n^{1-1/p}\, r_a^{-n_p}$ and using $N_p' + 1 \ge n_p$ (since $N_p' \le n_p \le N_p' + 1$), the power of $r_a$ is bounded, $r_a^{N_p'+1-n_p} \le 1$ as $r_a < 1$, giving:
\begin{equation}
\label{eq:small-weighted-schwartz}
  \abs{a_\varphi(x)}
  \;\le\; C(n, p, \varphi)\, (1+\abs{x-x_a})^{-M} , \qquad x \in \mathbb{R}^n .
\end{equation}
For the molecule property, at every exponent, fix $1 < \tilde q \le \infty$. Since $(1+\abs{x})^{-M} \in L^{\tilde q}(\mathbb{R}^n)$ for $M > n/\tilde q$, which holds as $M > n/p \ge n \ge n/\tilde q$, \eqref{eq:small-weighted-schwartz} gives $\norm{a_\varphi}_{\tilde q} \le C(n,p,\tilde q,\varphi)$, while $\abs{B(x_a,1)}^{1/\tilde q - 1/p}$ is a dimensional constant, and for $\abs{x-x_a} \ge 2$ the envelope passes to the annuli of $B(x_a,1)$ exactly as in part (i), giving condition (i) of a classical $(p,\tilde q,\varepsilon)$-molecule associated with $B(x_a,1)$, up to a constant depending only on $(n,p,\varepsilon,\varphi)$. The constants are independent of $\tilde q$ and of the exponent $q$ of the input atom, and the moments were verified above.

For part (iii), fix $\tilde M := n/p + 1$, so that $(1+\abs{x})^{-\tilde M} \in L^p(\mathbb{R}^n)$ since $\tilde M p > n$. For the small atom, \eqref{eq:small-weighted-schwartz} at $M = \tilde M$ gives $\abs{\varphi * a(x)} = \abs{a_\varphi(x)} \le C(n,p,\varphi)\, (1+\abs{x-x_a})^{-\tilde M}$ for every $x$, whence $\norm{\varphi * a}_p \le C(n,p,\varphi)$. For the large atom, on $2B_A$, Hölder's Inequality and Young's Inequality give $\norm{\varphi * A}_{L^p(2B_A)} \le \abs{2B_A}^{1/p-1/q}\, \norm{\varphi * A}_q \le \abs{2B_A}^{1/p-1/q}\, \norm{\varphi}_1\, \norm{A}_q \le 2^{n(1/p-1/q)}\, \norm{\varphi}_1$, using $\abs{B_A}^{1/p-1/q} \norm{A}_q \le 1$; on $\abs{x-x_A} \ge 2r_A$, the same tail bound underlying part (i) at $M = \tilde M$ gives $\abs{\varphi*A(x)} \le C(n,\varphi)\, (1+\abs{x-x_A})^{-\tilde M}$, using $\abs{B_A}^{1-1/p} \le \omega_n^{1-1/p}$ as $r_A \ge 1$, whence $\norm{\varphi*A}_{L^p(\abs{x-x_A}\ge 2r_A)} \le C(n,p,\varphi)$. Combining the two regions gives $\norm{\varphi*A}_p \le C(n,p,\varphi)$.
\end{proof}

\begin{lemma}
\label{lem:wavelet-LP} Let $1 < q < \infty$. For every $f \in L^q(\mathbb{R}^n)$:
\[
\norm[\bigg]{\biggl(\sum_{\lambda \in \Lambda} \abs{\pair{f}{\psi_\lambda}}^2\, \abs{\psi_\lambda}^2 \biggr)^{1/2}}_q \;\le\; c_q\, \norm{f}_q,
\]
where $c_q$ depends only on $n$, $q$, and the wavelet basis. The same bound holds, with the same constant, when $\Lambda$ is replaced by any of its subsets.
\end{lemma}

\begin{proof}
This bound is one inequality of the two-sided square function equivalence in the classical theorem that orthonormal wavelet bases with mild size and regularity assumptions are unconditional bases of $L^q(\mathbb{R}^n)$, $1 < q < \infty$ (see~\cite[Section~6.2]{meyer1992} and~\cite[Chapter~8]{wojtaszczyk1997}). The proof is a vector-valued Calderón–Zygmund argument: for finite $F \subseteq \Lambda$, the operator $f \mapsto \bigl( \pair{f}{\psi_\lambda}\, \psi_\lambda \bigr)_{\lambda \in F}$ has an $\ell^2(F)$-valued kernel whose size and Hörmander estimates hold with constants determined solely by $n$, the support constant $m$, $\max_e \norm{\psi^e}_\infty$, and a Hölder constant of the generators at any exponent $\delta \in (0, \min(r, 1)]$. These hypotheses are exactly satisfied by the compactly supported Daubechies system fixed above. The case $q = 2$ is Bessel's inequality, and the Calderón–Zygmund theorem for Banach space valued kernels (see, e.g., \cite[Chapter~V]{garcia-cuerva-rubio1985} or~\cite[Section~5.6]{Grafakos2014}) together with Marcinkiewicz interpolation and duality gives all $1 < q < \infty$, with constants uniform in $F$. The Monotone Convergence Theorem, applied along an exhaustion of $\Lambda$ by finite sets, then yields the lemma and the subset clause follows since the summed terms are non-negative. Only the range $1 < q \le 2$ is invoked in this paper, in the proof of Lemma~\ref{lem:annulus-estimate-hp}.
\end{proof}

\begin{lemma}
\label{lem:annulus-estimate-hp} Let $1 < q \le \infty$ and let $a$ be an $L^q$ $h^{p}$-atom associated with the ball $B = B(x_0, r_B)$. Then, with:
\[
  S_{\mathrm{Loc}}^{\psi}(a)(x) := \biggl( \sum_{\lambda \in \Lambda^0}
     \abs{\pair{a}{\psi_\lambda}}^2\,
     \abs{\psi_\lambda(x)}^2 \biggr)^{1/2},
\]
the mother block of $S_{\mathrm{Loc}}$ in~\eqref{eq:SLoc-def}, we have:
\[
  \int_{\mathbb{R}^n} S_{\mathrm{Loc}}^{\psi}(a)^{p}\, dx
  \;\le\; C(n, p, q, m, \text{wavelet basis}) \;<\; \infty,
\]
uniformly over all such atoms.
\end{lemma}

\begin{proof}
We follow the strategy of~\cite[Theorem~4.2.1, (G)$\implies$(B)]{elaoumari2025} with few key changes.\\
The case $2 \le q \le \infty$ reduces to $q = 2$. By Hölder's Inequality on $B$:
\[
\norm{a}_2 \le \abs{B}^{1/2 - 1/q}\, \norm{a}_q \le \abs{B}^{1/2 - 1/p},
\]
so $a$ is an $L^2$ $h^{p}$-atom associated with the same ball, with the same moment conditions if required, and the conclusion follows from the case $q = 2$ with a $q$-independent constant. We may thus assume $1 < q \le 2$ in what follows.\\
Let $C_0 := 2 + m\sqrt{n} > 1$, and decompose $\mathbb{R}^n$ into the inflated ball $\tilde{B} = B(x_0, C_0 r_B)$ and the concentric dyadic annuli $A_k = \{x \in \mathbb{R}^n : 2^k \le \abs{x - x_0}/(C_0 r_B) < 2^{k+1}\}$, $k \in \mathbb{N} \cup \{0\}$. Since these sets partition $\mathbb{R}^n$ and by the $p$-triangle Inequality, we have:
\[
\int_{\mathbb{R}^n} S_{\mathrm{Loc}}^{\psi}(a)^p \leq \int_{\Tilde{B}} S_{\mathrm{Loc}}^{\psi}(a)^p + \sum_{k=0}^{\infty} \int_{A_{k}} S_{\mathrm{Loc}}^{\psi}(a)^p.
\]
For the contribution from $\Tilde{B}$, by Hölder's Inequality:
\[
\int_{\Tilde{B}} S_{\mathrm{Loc}}^{\psi}(a)^p \leq \abs{\Tilde{B}}^{1-p/q} \norm{S_{\mathrm{Loc}}^{\psi}(a)}_q^{p}
\]
By the Littlewood-Paley Inequality for the wavelet system (Lemma~\ref{lem:wavelet-LP} above), together with the trivial pointwise domination of $S_{\mathrm{Loc}}^{\psi}(a)$ by the full square function, and the size condition on the atom:
\[
\norm*{S_{\mathrm{Loc}}^{\psi}(a)}_q \le c_q \norm*{a}_q \leq c_q\, \abs*{B}^{1/q-1/p}.
\]
Combining:
\[
\int_{\Tilde{B}} S_{\mathrm{Loc}}^{\psi}(a)^p \leq c_q^p\, C_0^{n(1-p/q)}.
\]
For the contribution from the annuli $A_k$, we have two cases:\\
\noindent For large-ball atoms, $r_B \ge 1$: If $\pair{a}{\psi_\lambda} \ne 0$ and $\psi_\lambda(x) \ne 0$, then $m Q(\lambda)$ meets both $B$ and $\{x\}$, so $\abs{x - x_0} \le r_B + \diam(m Q(\lambda)) \le r_B + m\sqrt{n}\, 2^{-j} \le r_B (1 + m\sqrt{n}) < C_0 r_B$, using $2^{-j} \le 1 \le r_B$. Hence $S_{\mathrm{Loc}}^{\psi}(a)$ vanishes on every $A_k$, and the annulus sum is empty.\\
\noindent For small-ball atoms, $r_B < 1$ and $a$ has vanishing moments up to order $N_p$: Fix a Hölder exponent $\delta$ with:
\[
n_p - N_p \;<\; \delta \;\le\; \min(r - N_p,\, 1),
\]
which is possible since $N_p = \lfloor n_p \rfloor \le n_p < r$ and $n_p - N_p < 1$. Then: $\psi^e \in C^{N_p, \delta}$ for every $e$. For $x \in A_k$, let:
\[
\Lambda_k(x) := \{\lambda \in \Lambda^0 : x \in m Q(\lambda),\ m Q(\lambda) \cap B \ne \emptyset\}
\]
be the set of indices contributing to $S_{\mathrm{Loc}}^{\psi}(a)(x)$. For $\lambda \in \Lambda_k(x)$, the cube $m Q(\lambda)$ joins $x$ to $B$, so $m\sqrt{n}\, 2^{-j} \ge \abs{x - x_0} - r_B \ge (C_0 2^k - 1) r_B \ge (C_0 - 1)\, 2^k r_B$, i.e.\ the scales are confined to the window:
\begin{equation}
\label{eq:scale-window}
1 \;\ge\; 2^{-j} \;\ge\; c\, r_B\, 2^k,
  \qquad c := \frac{C_0 - 1}{m \sqrt{n}} ,
\end{equation}
with at most $C(m, n)$ indices per scale and type. For the coefficients: expanding $\psi_\lambda$ at $x_0$ to order $N_p$, we get:
\[
  \psi_\lambda(y) = \sum_{\abs{\nu} \le N_p}
     \frac{(y - x_0)^\nu}{\nu!}\,
     \partial^\nu \psi_\lambda(x_0) + R_\lambda(y),
  \qquad
  \abs{R_\lambda(y)} \le C(\psi, \delta)\, 2^{nj/2}\,
     2^{j(N_p + \delta)}\, \abs{y - x_0}^{N_p + \delta},
\]
the remainder bound from the $C^{N_p, \delta}$ regularity of the rescaled basis. The moments of $a$ annihilate the polynomial part, so, using $\norm{a}_1 \le \abs{B}^{1 - 1/p}$ from~\eqref{eq:atom-l1}, and $y \in B$:
\[
  \abs{\pair{a}{\psi_\lambda}}
  = \abs[\Big]{\int_B a(y)\, R_\lambda(y)\, dy}
  \;\lesssim\; 2^{j(N_p + \delta + n/2)}\,
     r_B^{\,N_p + \delta}\, \abs{B}^{1 - 1/p}
  \;\lesssim\; 2^{j(N_p + \delta + n/2)}\,
     r_B^{\,N_p + \delta - n(1/p - 1)} .
\]
By the scale window~\eqref{eq:scale-window}, $2^{j(N_p + \delta + n/2)} \le (c\, r_B 2^k)^{-(N_p + \delta + n/2)}$, hence, uniformly over $\lambda \in \Lambda_k(x)$:
\[
  \abs{\pair{a}{\psi_\lambda}}
  \;\lesssim\; 2^{-k(N_p + \delta + n/2)}\,
     r_B^{\,-n/2 - n(1/p - 1)} .
\]
On the other hand, with $\abs{\psi_\lambda(x)} \lesssim 2^{nj/2}$ and $C(m, n)$ indices per scale in the window:
\[
  \biggl( \sum_{\lambda \in \Lambda_k(x)}
     \abs{\psi_\lambda(x)}^2 \biggr)^{1/2}
  \;\le\; \sum_{\lambda \in \Lambda_k(x)}
     \abs{\psi_\lambda(x)}
  \;\lesssim\; \sum_{j :\, 2^{-j} \ge c r_B 2^k} 2^{nj/2}
  \;\lesssim\; 2^{-nk/2}\, r_B^{-n/2},
\]
the geometric sum dominated by its largest term. Combining the two bounds:
\[
  S_{\mathrm{Loc}}^{\psi}(a)(x)^p \;\lesssim\; 2^{-k p (N_p + \delta + n)}\,
     r_B^{-n},
  \qquad x \in A_k ,
\]
and, with $\abs{A_k} \le \omega_n (2^{k+1} C_0 r_B)^n$:
\[
  \int_{A_k} S_{\mathrm{Loc}}^{\psi}(a)^p
  \;\lesssim\; 2^{k n}\, 2^{-k p(N_p + \delta + n)}
  \;=\; 2^{-k p (N_p + \delta - n_p)} .
\]
Since $\delta > n_p - N_p$, the exponent is positive and the sum over $k$ converges. The total bound is $\int_{\mathbb{R}^n} S_{\mathrm{Loc}}^{\psi}(a)^p \le C(n, p, m, \text{wavelet basis})$, as claimed.
\end{proof}

\begin{lemma}
\label{lem:lowfreq-atoms} Let $0 < p \le 1$ and let $\varphi \in \mathcal S(\mathbb{R}^n)$ with $\supp(\hat\varphi) \subseteq \overline{B(0, R)}$. Let $f \in \mathcal{S}'(\mathbb{R}^n)$ be such that $u := \varphi * f \in L^p(\mathbb{R}^n)$. For $k \in \mathbb{Z}^n$, let $C_k = k + [-\tfrac12, \tfrac12)^n$ (translate of the unit cube), $B_k = B(k, \max(1, \sqrt{n}))$ (a ball of radius at least $1$ containing $\overline{C_k}$), and set:
\[
  \mu_k \;:=\; \abs{B_k}^{1/p}\, \norm{u}_{L^\infty(C_k)},
  \qquad
  A_k \;:=\; \mu_k^{-1}\, u\, \mathbf{1}_{C_k}
  \quad (\text{when } \mu_k \ne 0).
\]
Then each $A_k$ is an $L^\infty$ $h^{p}$-atom associated with the ball $B_k$ of radius $\ge 1$ (in particular, no moment conditions are required), the decomposition $u = \sum_{k \in \mathbb{Z}^n} \mu_k A_k$ holds pointwise and in $\mathcal{S}'(\mathbb{R}^n)$ (the terms with $\mu_k = 0$ being omitted), and:
\[
  \sum_{k \in \mathbb{Z}^n} \abs{\mu_k}^p
  \;\le\; C(n, p, R)\, \norm{\varphi * f}_p^p .
\]
Hence $\varphi * f \in h^p(\mathbb{R}^n)$ with $\norm{\varphi * f}_{h^{p}, \mathrm{at}} \le C(n, p, R)\, \norm{\varphi * f}_p$.
\end{lemma}

\begin{proof}
Since $\hat{u} = \hat\varphi\, \hat{f}$ is a tempered distribution supported in $\overline{B(0, R)}$, the Paley-Wiener-Schwartz Theorem makes $u$ a smooth function of polynomial growth. In particular, $u$ is continuous, so $\norm{u}_{L^\infty(C_k)} = \sup_{C_k} \abs{u}$, and $u \varphi' \in L^1(\mathbb{R}^n)$ for every $\varphi' \in \mathcal S(\mathbb{R}^n)$.

The atoms are immediate: $\supp(A_k) \subseteq \overline{C_k} \subseteq B_k$, $\norm{A_k}_\infty = \mu_k^{-1} \norm{u}_{L^\infty(C_k)} = \abs{B_k}^{-1/p}$, and $B_k$ has radius $\max(1, \sqrt{n}) \ge 1$, so no moment condition applies.

For the convergence, the cubes $\{C_k\}$ partition $\mathbb{R}^n$, so $\sum_k \mu_k A_k = u \sum_k \mathbf{1}_{C_k} = u$ pointwise. For the convergence in $\mathcal{S}'(\mathbb{R}^n)$: for $\varphi' \in \mathcal S(\mathbb{R}^n)$ and finite $G \subseteq \mathbb{Z}^n$, we have:
\[
  \pair{\sum_{k \in G} \mu_k A_k}{\varphi'}
  \;=\; \int_{\bigcup_{k \in G} C_k} u\, \varphi'\, dx
  \;\xrightarrow[G \uparrow \mathbb{Z}^n]{}\;
  \int_{\mathbb{R}^n} u\, \varphi'\, dx
  \;=\; \pair{u}{\varphi'},
\]
by the Dominated Convergence Theorem, the dominating function $\abs{u \varphi'}$ being integrable by the polynomial growth of $u$.

For the coefficient bound, we invoke the Plancherel-Pólya Inequality~\cite[Sections~1.3.2--1.3.3]{triebel1983}: for every $0 < p \le \infty$, there is $C_1 = C_1(n, p, R)$ such that every $v \in L^p(\mathbb{R}^n) \cap \mathcal{S}'(\mathbb{R}^n)$ with $\supp(\hat{v}) \subseteq \overline{B(0, R)}$ satisfies:
\begin{equation}
\label{eq:pp}
  \sum_{k \in \mathbb{Z}^n}\, \sup_{x \in C_k} \abs{v(x)}^p
  \;\le\; C_1\, \norm{v}_p^p .
\end{equation}
Applying~\eqref{eq:pp} to $v = u$, and noting that $\abs{B_k} = \omega_n \max(1, \sqrt{n})^n$ for every $k$:
\[
  \sum_{k} \abs{\mu_k}^p
  \;=\; \omega_n \max(1, \sqrt{n})^n
     \sum_{k} \norm{u}_{L^\infty(C_k)}^p
  \;\le\; \omega_n \max(1, \sqrt{n})^n\, C_1\,
     \norm{u}_p^p ,
\]
which is the claim with $C(n, p, R) = \omega_n \max(1, \sqrt{n})^n\, C_1$.

To conclude, by the uniform atomic bound~\eqref{eq:atom-hp-bound} and $p$-sub-additivity, the series $\sum_k \mu_k A_k$ converges unconditionally in $h^{p}$, and, its terms being $h^{p}$-atoms with $\ell^p$ coefficients, $\norm[\big]{\sum_k \mu_k A_k}_{h^{p}, \mathrm{at}}^p \le \sum_k \abs{\mu_k}^p$ by definition of the atomic quasi-norm. Since $h^{p}$ embeds continuously in $\mathcal{S}'(\mathbb{R}^n)$ and the series converges to $u$ in $\mathcal{S}'(\mathbb{R}^n)$, its $h^{p}$-sum is $u$. Hence $u \in h^{p}$ with $\norm{u}_{h^{p}, \mathrm{at}} \le C(n, p, R)^{1/p}\, \norm{u}_p$.
\end{proof}

\begin{lemma}
\label{lem:father-coeff-lp} Let $0 < p \le 1$ and let $f \in \mathcal{S}'(\mathbb{R}^n)$ satisfy condition \textup{(G)}: $f - \varphi * f \in H^p(\mathbb{R}^n)$ and $\varphi * f \in L^p(\mathbb{R}^n)$ for every admissible Schwartz $\varphi$. Let $\varphi_0$ be the fixed band-limited Schwartz function exhibited in~\eqref{eq:admiss-p-s}, with $\supp(\hat\varphi_0) \subseteq \overline{B(0, R)}$. Then, the father wavelet coefficients $\beta_k = \pair{f}{\phi_k}$, interpreted via the pairing convention \textup{(P3)} of Subsection~\ref{ssec:wavelet}, satisfy:
\[
  \sum_{k \in \mathbb{Z}^n} \abs{\beta_k}^p
  \;\le\; C \Bigl( \norm{f - \varphi_0 * f}_{H^{p}}^p
     + \norm{\varphi_0 * f}_p^p \Bigr)
  \;<\; \infty,
\]
with $C = C(n, p, m, R, \text{wavelet basis})$.
\end{lemma}

\begin{proof}
Since $\varphi_0$ is an admissible Schwartz function, condition \textup{(G)} gives, by assumption:
\[
  g := f - \varphi_0 * f \in H^p(\mathbb{R}^n),
  \qquad
  u := \varphi_0 * f \in L^p(\mathbb{R}^n),
\]
and $u \in C^\infty(\mathbb{R}^n)$ with polynomially bounded derivatives, as the convolution of a tempered distribution with a Schwartz function, and $\hat{u} = \hat\varphi_0 \hat{f}$ is supported in $\overline{B(0, R)}$. Fix an atomic decomposition~\cite{latter1978}, \cite[Chapter~III]{garcia-cuerva-rubio1985}:
\begin{equation}
\label{eq:atomic-dec}
  g = \sum_i \mu_i\, a_i
  \quad (\text{convergence in } H^{p}),
  \qquad
  \sum_i \abs{\mu_i}^p \le C_{\mathrm{dec}}\,
     \norm{g}_{H^{p}}^p ,
\end{equation}
into $L^2$ $H^{p}$-atoms: each $a_i$ is supported in a ball $B_i = B(x_i, r_i)$, satisfies $\norm{a_i}_2 \le \abs{B_i}^{1/2 - 1/p}$, and $\int y^\nu a_i(y)\, dy = 0$, for all $\abs{\nu} \le N_p$.

Set:
\[
\delta := \min(r - N_p,\, 1) \in (0, 1],
  \qquad \text{so that} \quad N_p + \delta > n_p ,
\]
the first assertion because $r > n_p \ge N_p$, and the second because $N_p + \delta$ equals either $r > n_p$ or $N_p + 1 > n_p$. \begin{equation}
\label{eq:phik-reg}
  \phi_k \in C^{N_p, \delta}_c(\mathbb{R}^n),
  \qquad
  \norm{\phi_k}_{C^{N_p, \delta}} =: C_\phi < \infty
  \quad \text{uniformly in } k .
\end{equation}
Indeed, writing $\phi_k(x) = \prod_{i=1}^n \phi(x_i - k_i)$: for $\abs{\nu} \le N_p$ each component satisfies $\nu_i \le N_p < r$, so $\partial^\nu \phi_k(x) = \prod_i \phi^{(\nu_i)}(x_i - k_i)$ exists and is continuous, and for $\abs{\nu} = N_p$, each factor $\phi^{(\nu_i)}$ belongs to $C^{\,r - \nu_i}(\mathbb{R}) \subseteq C^{0, \delta}(\mathbb{R})$, as $r - \nu_i \ge r - N_p \ge \delta$, and a finite product of bounded, compactly supported, $\delta$-Hölder functions is $\delta$-Hölder. The uniformity in $k$ is translation invariance.

Since $N_p + \delta > n_p$ and $\phi_k$ has compact support, $\phi_k$ belongs to the dual space of $H^{p}$: for $0 < p < 1$, $(H^{p})^* \cong \dot\Lambda_{n_p}(\mathbb{R}^n)$ and $C^{N_p, \delta}_c \hookrightarrow \dot\Lambda_{n_p}$, while for $p = 1$, $(\mathrm{H}^1)^* = \mathrm{BMO}$ and $L^\infty_c \subseteq \mathrm{BMO}$~\cite{feffermanstein1972, taibleson1980}. Hence, the functional $h \mapsto \pair{h}{\phi_k}$, which on $L^2 \cap H^{p}$ is the integral $\int h\, \overline{\phi_k}$, extends continuously to $H^{p}$, with norm bounded uniformly in $k$ by~\eqref{eq:phik-reg}. Following convention \textup{(P3)}, define:
\[
  \beta_k^{(1)} := \pair{g}{\phi_k}
  \quad (\text{duality pairing}),
  \qquad
  \beta_k^{(2)} := \int_{\mathbb{R}^n} u(x)\, \overline{\phi_k(x)}\, dx
  \quad (\text{absolutely convergent}),
\]
the second because $u$ is continuous, hence bounded on the compact set $\supp(\phi_k) \subseteq m Q_k$. Then $\beta_k = \beta_k^{(1)} + \beta_k^{(2)}$, by the very definition~\eqref{eq:G-pairing} of convention \textup{(P3)}, and by $p$-sub-additivity:
\begin{equation}
\label{eq:split}
  \abs{\beta_k}^p
  \;\le\; \abs{\beta_k^{(1)}}^p + \abs{\beta_k^{(2)}}^p .
\end{equation}

There is $C_2 = C_2(n, p, m, \text{wavelet basis})$ such that for every $L^2$ $H^{p}$-atom $a$ on a ball $B = B(x_a, r_a)$:
\begin{equation}
\label{eq:atom-est}
  \sum_{k \in \mathbb{Z}^n} \abs{\pair{a}{\phi_k}}^p \;\le\; C_2,
  \qquad \text{and} \qquad
  \sup_{k \in \mathbb{Z}^n} \abs{\pair{a}{\phi_k}} \;\le\; C_2 ,
\end{equation}
where, $a$ being an element of $L^2 \cap H^{p}$, the pairing $\pair{a}{\phi_k} = \int a\, \overline{\phi_k}$ is an honest (absolutely convergent) integral, on which the duality functional $\pair{\cdot}{\phi_k}_*$ agrees.\\
Let $K := \{k \in \mathbb{Z}^n : m Q_k \cap B \ne \emptyset\}$. Since $\supp(\phi_k) \subseteq m Q_k$, we have $\pair{a}{\phi_k} = 0$ for $k \notin K$. As $m Q_k \subseteq B(k, (m+1)\sqrt{n})$, every $k \in K$ satisfies $\abs{k - x_a} \le r_a + (m+1)\sqrt{n}$, so:
\begin{equation}
\label{eq:counting}
  \# K \;\le\; C(m, n)\, (1 + r_a)^n .
\end{equation}

By~\eqref{eq:counting}, $\# K \le C(m,n)\, 2^n r_a^n \le C'(m,n)\, \abs{B}$. Since $\{\phi_k\}_{k \in \mathbb{Z}^n}$ is an orthonormal family in $L^2$, Bessel's inequality gives $\sum_k \abs{\pair{a}{\phi_k}}^2 \le \norm{a}_2^2 \le \abs{B}^{1 - 2/p}$. By Hölder's Inequality on $K$ with exponents $2/p$ and $(1 - p/2)^{-1}$:
\[
  \sum_{k \in K} \abs{\pair{a}{\phi_k}}^p
  \;\le\; (\# K)^{1 - p/2}
     \Bigl( \sum_{k} \abs{\pair{a}{\phi_k}}^2 \Bigr)^{p/2}
  \;\le\; C'(m,n)^{1 - p/2}\, \abs{B}^{1 - p/2}\,
     \abs{B}^{p/2 - 1}
  \;=\; C'(m,n)^{1 - p/2} .
\]
Individually, $\abs{\pair{a}{\phi_k}} \le \norm{a}_2 \norm{\phi_k}_2 \le \abs{B}^{1/2 - 1/p} \le \omega_n^{1/2 - 1/p}$, since $\abs{B} \ge \omega_n$ and $1/2 - 1/p \le 0$.

By~\eqref{eq:counting}, $\# K \le C(m,n)\, 2^n$. Let $P_k$ be the Taylor polynomial of degree $N_p$ of $\phi_k$ at $x_a$. By~\eqref{eq:phik-reg} and the Taylor remainder in Hölder form, for $x \in B$:
\[
  \abs{\phi_k(x) - P_k(x)}
  \;\le\; C_{n, N_p}\, C_\phi\,
     \abs{x - x_a}^{N_p + \delta}
  \;\le\; C_{n, N_p}\, C_\phi\, r_a^{\,N_p + \delta} .
\]
The atom's moments give $\int a\, \overline{P_k} = 0$, so, using $\norm{a}_1 \le \abs{B}^{1 - 1/p}$ from~\eqref{eq:atom-l1}, we have:
\begin{align*}
  \abs{\pair{a}{\phi_k}}
  = \abs[\Big]{\int_B a(x)\,
     \overline{\bigl( \phi_k(x) - P_k(x) \bigr)}\, dx}
  &\;\le\; C_{n, N_p}\, C_\phi\, r_a^{\,N_p + \delta}\,
     \abs{B}^{1 - 1/p} \\
  &\;=\; C_{n, N_p}\, C_\phi\, \omega_n^{1 - 1/p}\,
     r_a^{\,N_p + \delta - n_p}
  \;\le\; C,
\end{align*}
since $N_p + \delta > n_p$ and $0 < r_a < 1$. Summing over the at most $C(m,n)\, 2^n$ indices of $K$ proves~\eqref{eq:atom-est} in this case as well.

Since partial sums of~\eqref{eq:atomic-dec} converge to $g$ in $H^{p}$ and $\pair{\cdot}{\phi_k}_*$ is continuous on $H^{p}$:
\[
  \beta_k^{(1)} = \pair{g}{\phi_k}_*
  = \sum_i \mu_i\, \pair{a_i}{\phi_k} .
\]
By $p$-sub-additivity, Tonelli's Theorem, and~\eqref{eq:atom-est}, we have:
\begin{equation}
\label{eq:term1}
  \sum_{k \in \mathbb{Z}^n} \abs{\beta_k^{(1)}}^p
  \;\le\; \sum_{k} \sum_i \abs{\mu_i}^p
     \abs{\pair{a_i}{\phi_k}}^p
  \;=\; \sum_i \abs{\mu_i}^p
     \sum_k \abs{\pair{a_i}{\phi_k}}^p
  \;\le\; C_2\, C_{\mathrm{dec}}\, \norm{g}_{H^{p}}^p .
\end{equation}

For the $L^p$ term, since $\supp(\phi_k) \subseteq m Q_k$ and $\norm{\phi_k}_\infty = \norm{\phi}_{L^\infty(\mathbb{R})}^n$:
\[
  \abs{\beta_k^{(2)}}
  \;\le\; \norm{\phi}_{L^\infty(\mathbb{R})}^n
     \int_{m Q_k} \abs{u(x)}\, dx
  \;\le\; \norm{\phi}_{L^\infty(\mathbb{R})}^n\, m^n\,
     \sup_{m Q_k} \abs{u} .
\]
The cube $m Q_k$ is covered by the translates of the unit cube $C_{k'}$, $k' \in N(k) := \{k' : C_{k'} \cap m Q_k \ne \emptyset\}$, with $\# N(k) \le C(m, n)$ and each $k'$ belonging to $N(k)$ for at most $C(m, n)$ values of $k$. Hence, $\sup_{m Q_k} \abs{u}^p \le \sum_{k' \in N(k)} \sup_{C_{k'}} \abs{u}^p$, and summing over $k$:
\begin{equation}
\label{eq:term2}
  \sum_{k \in \mathbb{Z}^n} \abs{\beta_k^{(2)}}^p
  \;\le\; \norm{\phi}_{L^\infty(\mathbb{R})}^{np}\, m^{np}\,
     C(m, n)^2 \sum_{k' \in \mathbb{Z}^n}
     \sup_{C_{k'}} \abs{u}^p
  \;\le\; C(n, p, m, R, \text{wavelet basis})\,
     \norm{u}_p^p ,
\end{equation}
the last step by the Plancherel-Pólya inequality~\eqref{eq:pp}, applicable since $\hat{u}$ is supported in $\overline{B(0, R)}$ and $u \in L^p(\mathbb{R}^n) \cap \mathcal{S}'(\mathbb{R}^n)$. This application of~\eqref{eq:pp} is the step requiring the band-limited choice of $\varphi_0$: for general $u \in L^p(\mathbb{R}^n)$ with $p < 1$, the left-hand side of~\eqref{eq:term2} need not be finite.

Combining~\eqref{eq:split}, \eqref{eq:term1}, and~\eqref{eq:term2}:
\[
  \sum_{k \in \mathbb{Z}^n} \abs{\beta_k}^p
  \;\le\; C_2\, C_{\mathrm{dec}}\, \norm{g}_{H^{p}}^p
  + C\, \norm{u}_p^p
  \;<\; \infty,
\]
which is the assertion.
\end{proof}

\begin{lemma}
\label{lem:S0-sublinear} Let $0 < p \le 1 < q \le \infty$ and let $f = \sum_i \mu_i a_i$, with convergence in $\mathcal{S}'(\mathbb{R}^n)$, where each $a_i$ is an $L^q$ $h^{p}$-atom and $\sum_i \abs{\mu_i}^p < \infty$. Then, for every $\lambda \in \Lambda^0$, the series defining the wavelet coefficients of $f$ under convention \textup{(P2)} of Subsection~\ref{ssec:wavelet}:
\[
  \pair{f}{\psi_\lambda}
  \;=\; \sum_i \mu_i\, \pair{a_i}{\psi_\lambda},
\]
converges absolutely. Moreover, pointwise on $\mathbb{R}^n$:
\[
  S_{\mathrm{Loc}}^{\psi}(f)(x) \;\le\; \sum_i \abs{\mu_i}\, S_{\mathrm{Loc}}^{\psi}(a_i)(x),
\]
and hence, by $p$-sub-additivity and Tonelli's Theorem:
\[
  S_{\mathrm{Loc}}^{\psi}(f)(x)^p \;\le\; \sum_i \abs{\mu_i}^p\, S_{\mathrm{Loc}}^{\psi}(a_i)(x)^p,
  \qquad
  \int_{\mathbb{R}^n} S_{\mathrm{Loc}}^{\psi}(f)^p\, dx
  \;\le\; \sum_i \abs{\mu_i}^p
     \int_{\mathbb{R}^n} S_{\mathrm{Loc}}^{\psi}(a_i)^p\, dx .
\]
\end{lemma}

\begin{proof}
Fix $\lambda \in \Lambda^0$. We claim:
\begin{equation}
\label{eq:unif-atom-pairing}
  \sup \bigl\{ \abs{\pair{a}{\psi_\lambda}} \;:\;
     a \text{ an } L^q \ h^{p}\text{-atom} \bigr\}
  \;=:\; C_\lambda \;<\; \infty .
\end{equation}
Every $L^q$ $h^{p}$-atom $a$ on a ball $B_a$ satisfies $\norm{a}_1 \le \abs{B_a}^{1 - 1/p}$ by~\eqref{eq:atom-l1}. If $a$ is a large-ball atom ($r_a \ge 1$), then $\abs{B_a} \ge \omega_n$ and $1 - 1/p \le 0$ give:
\[
  \abs{\pair{a}{\psi_\lambda}}
  \;\le\; \norm{\psi_\lambda}_\infty\, \norm{a}_1
  \;\le\; \omega_n^{1 - 1/p}\, \norm{\psi_\lambda}_\infty .
\]
If $a$ is a small-ball atom ($r_a < 1$), let $P$ be the Taylor polynomial of degree $N_p$ of $\psi_\lambda$ at the center $x_a$. The basis regularity $r > n_p$ makes $\psi_\lambda \in C^{N_p, \delta}_c(\mathbb{R}^n)$ with $\delta = \min(r, N_p + 1) - N_p$ (as in convention \textup{(P3)}), so the moment conditions of $a$ and the Hölder form of the Taylor remainder~\eqref{eq:taylor-holder} give:
\[
  \abs{\pair{a}{\psi_\lambda}}
  = \abs[\Big]{ \int_{B_a} a\,
     \overline{(\psi_\lambda - P)} }
  \;\le\; C_{n, N_p}\,
     \norm{\psi_\lambda}_{C^{N_p, \delta}}\,
     r_a^{\,N_p + \delta}\, \abs{B_a}^{1 - 1/p}
  \;\le\; C_{n, N_p}\, \omega_n^{1 - 1/p}\,
     \norm{\psi_\lambda}_{C^{N_p, \delta}} ,
\]
since $N_p + \delta > n_p$ and $r_a < 1$. This estimate proves~\eqref{eq:unif-atom-pairing}. Here $C_\lambda$ depends on $\lambda$ (through the rescaled norms of $\psi_\lambda$), which is harmless here. Since $\ell^p \hookrightarrow \ell^1$ for $p \le 1$, the coefficient series converges absolutely: $\sum_i \abs{\mu_i}\, \abs{\pair{a_i}{\psi_\lambda}} \le C_\lambda \bigl( \sum_i \abs{\mu_i}^p \bigr)^{1/p} < \infty$.

Now fix $x \in \mathbb{R}^n$ and $M \in \mathbb{N}$, and consider the vectors $v_i := \bigl( \pair{a_i}{\psi_\lambda}\, \psi_\lambda(x) \bigr)_{\lambda \in \Lambda^0} \in \ell^2(\Lambda^0)$. Each has $\norm{v_i}_{\ell^2} = S_{\mathrm{Loc}}^{\psi}(a_i)(x)$, finite for a.e.\ $x$ by Lemma~\ref{lem:annulus-estimate-hp}, hence, the atoms being countably many, all are finite for a.e.\ $x$ at once (and at the exceptional points the claimed inequality holds trivially, its right-hand side being infinite). By the triangle inequality in $\ell^2(\Lambda^0)$:
\[
  \Biggl( \sum_{\lambda \in \Lambda^0}
     \Bigl| \sum_{i \le M} \mu_i\,
     \pair{a_i}{\psi_\lambda}\, \psi_\lambda(x) \Bigr|^2
  \Biggr)^{1/2}
  \;=\; \norm[\Big]{ \sum_{i \le M} \mu_i v_i }_{\ell^2}
  \;\le\; \sum_{i \le M} \abs{\mu_i}\, \norm{v_i}_{\ell^2}
  \;\le\; \sum_{i} \abs{\mu_i}\, S_{\mathrm{Loc}}^{\psi}(a_i)(x) .
\]
To pass to the limit $M \to \infty$, we use Fatou's Lemma on $\Lambda^0$ with counting measure.\\
For each fixed $\lambda$, the absolute convergence gives $\sum_{i \le M} \mu_i \pair{a_i}{\psi_\lambda}\, \psi_\lambda(x) \to \pair{f}{\psi_\lambda}\, \psi_\lambda(x)$ as $M \to \infty$, so:
\begin{align*}
S_{\mathrm{Loc}}^{\psi}(f)(x)^2
&= \sum_{\lambda \in \Lambda^0} \lim_{M \to \infty} \Bigl| \sum_{i \le M} \mu_i \pair{a_i}{\psi_\lambda}\, \psi_\lambda(x) \Bigr|^2 \\
&\le \liminf_{M \to \infty} \sum_{\lambda \in \Lambda^0} \Bigl| \sum_{i \le M} \mu_i \pair{a_i}{\psi_\lambda}\, \psi_\lambda(x) \Bigr|^2
\;\le\; \Bigl( \sum_i \abs{\mu_i}\, S_{\mathrm{Loc}}^{\psi}(a_i)(x) \Bigr)^2 .
\end{align*}
Taking square roots proves the pointwise inequality. Finally, $p$-sub-additivity of $t \mapsto t^p$ applied to the countable non-negative sum gives $S_{\mathrm{Loc}}^{\psi}(f)(x)^p \le \sum_i \abs{\mu_i}^p S_{\mathrm{Loc}}^{\psi}(a_i)(x)^p$, and integrating with Tonelli's Theorem concludes the argument.
\end{proof}

\begin{lemma}
\label{lem:atom-Lp} Let $0 < p \le 1 < q \le \infty$. Every $L^q$ $h^{p}$-atom $a$, associated with a ball or a cube $K$, satisfies $\norm{a}_p \le 1$.
\end{lemma}

\begin{proof}
Suppose first $q < \infty$. By Hölder's Inequality with exponent $q/p > 1$ on $K$ and the size condition:
\[
  \int_K \abs{a}^p
  \;\le\; \abs{K}^{1 - p/q}\, \norm{a}_q^p
  \;\le\; \abs{K}^{1 - p/q}\, \abs{K}^{p/q - 1}
  \;=\; 1 .
\]
If $q = \infty$, then directly $\int_K \abs{a}^p \le \abs{K}\, \norm{a}_\infty^p \le \abs{K}\, \abs{K}^{-1} = 1$, no moment condition used, the bound exactly $1$.
\end{proof}

\begin{lemma}
\label{lem:moment-correction} Let $L \in \mathbb{N} \cup \{0\}$. For any prescribed values $m_\nu \in \mathbb{R}$, $\abs{\nu} \le L$, there exists $v \in C^\infty_c(\mathbb{R}^n)$ with $\supp(v) \subseteq \overline{B(0, 1)}$ and $\int_{\mathbb{R}^n} y^\nu\, v(y)\, dy = m_\nu$ for all $\abs{\nu} \le L$. The construction is linear in $(m_\nu)$, so every norm of $v$ is controlled by $\max_\nu \abs{m_\nu}$.
\end{lemma}

\begin{proof}
Fix $\chi \in C^\infty_c(\overline{B(0, 1)})$, $\chi \ge 0$, $\chi \not\equiv 0$, and seek $v = \sum_{\abs{\mu} \le L} c_\mu\, y^\mu \chi$. The moment conditions form the linear system $G c = m$ with $G_{\nu, \mu} = \int y^{\nu + \mu} \chi\, dy$, the Gram matrix of the monomials in $L^2(\chi\, dy)$: it is positive definite, since $c^{\mathsf T} G c = \int \bigl( \sum_\mu c_\mu y^\mu \bigr)^2 \chi\, dy > 0$ for $c \ne 0$, a non-zero polynomial vanishing only on a null set while $\chi$ is positive on an open set. Hence $c = G^{-1} m$.
\end{proof}

\begin{theorem}
\label{thm:hp-wavelet} \textit{(for $p = 1$, see~\cite[Theorem~4.2.1]{elaoumari2025}.)} With the wavelet basis fixed in the Conventions and the coefficients $\alpha(\lambda) = \pair{f}{\psi_\lambda}$, $\beta_k = \pair{f}{\phi_k}$ understood via the pairing conventions \textup{(P1)}-\textup{(P3)} of Subsection~\ref{ssec:wavelet}, for $f \in \mathcal{S}'(\mathbb{R}^n)$ the following are equivalent definitions of $f \in h^p(\mathbb{R}^n)$:
\begin{enumerate}[(A)]
\item \(\displaystyle \sup_{\substack{F \subseteq \mathbb{Z}^n \sqcup \Lambda^0 \\ \text{finite}}}\ \sup_{\varepsilon = \pm 1}\ \norm[\bigg]{\sum_{k \in F \cap \mathbb{Z}^n} \varepsilon(k)\, \beta_k\, \phi_k \;+ \sum_{\lambda \in F \cap \Lambda^0} \varepsilon(\lambda)\, \alpha(\lambda)\, \psi_\lambda }_p < \infty .\)
\item $S_{\mathrm{Loc}}(f) \in L^p(\mathbb{R}^n)$, the square function of~\eqref{eq:SLoc-def}.
\item $S_{\mathrm{Loc}}^R(f) \in L^p(\mathbb{R}^n)$.
\item $S_{\mathrm{Loc}}^Q(f) \in L^p(\mathbb{R}^n)$.
\item $f = \sum_j \lambda_{j}\, a_j$ in $\mathcal{S}'(\mathbb{R}^n)$, where the $a_j$'s are $L^q$ $h^{p}$-atoms associated with cubes, and $\sum_j \abs{\lambda_{j}}^p < \infty$, for some $1 < q \leq \infty$.
\item $f = \sum_j \lambda_{j}\, a_j$ in $\mathcal{S}'(\mathbb{R}^n)$, where the $a_j$'s are $L^q$ $h^{p}$-atoms associated with balls, and $\sum_j \abs{\lambda_{j}}^p < \infty$, for some $1 < q \leq \infty$.
\item For every admissible Schwartz function $\varphi \in \mathcal S(\mathbb{R}^n)$, that is, $\int \varphi = 1$ and $\int x^\nu \varphi\, dx = 0$ for $0 < \abs{\nu} \le N_p$, we have: $f - \varphi * f \in H^p(\mathbb{R}^n)$ and $\varphi * f \in L^p(\mathbb{R}^n)$.
\end{enumerate}
Also, the associated quasi-norms are equivalent, with constants depending only on $(n, p, q,$ wavelet basis) (see Corollary~\ref{cor:equiv-norms}).
\end{theorem}

\begin{proof}

We prove the implications:
\[
  \textup{(B)} \implies \textup{(C)} \implies \textup{(D)}
  \implies \textup{(E)} \implies \textup{(F)} \implies
  \textup{(G)} \implies \textup{(B)},
\]
together with the bridges \textup{(A)} $\implies$ \textup{(B)} and \textup{(F)} $\implies$ \textup{(A)}, establishing the equivalence of \textup{(A)}-\textup{(G)}. Write $C_1$ for the supremum in \textup{(A)}, $C_2, C_3, C_4$ for the $L^p$-quasi-norms of $S_{\mathrm{Loc}}(f), S_{\mathrm{Loc}}^R(f), S_{\mathrm{Loc}}^Q(f)$ of \textup{(B)}, \textup{(C)}, \textup{(D)}, and $C_5, C_6$ for $\bigl( \sum_j \abs{\lambda_j}^p \bigr)^{1/p}$ in \textup{(E)}, \textup{(F)}. Each implication propagates these quantitatively, the chain assembling into Corollary~\ref{cor:equiv-norms}.
\paragraph{\textbf{\boldmath (A) $\implies$ (B)}} Let $F \subseteq \mathbb{Z}^n \sqcup \Lambda^0$ be finite. For each $x \in \mathbb{R}^n$, Lemma~\ref{lem:khinchin-p} applied to the combined scalars $\{\beta_k \phi_k(x)\}_{k \in F \cap \mathbb{Z}^n} \cup \{\alpha(\lambda) \psi_\lambda(x)\}_{\lambda \in F \cap \Lambda^0}$ gives, writing $f_{F, \varepsilon}$ for the signed combined sum of condition \textup{(A)} at $(F, \varepsilon)$:
\[
  \biggl( \sum_{k \in F \cap \mathbb{Z}^n} \abs{\beta_k}^2 \abs{\phi_k(x)}^2
     + \sum_{\lambda \in F \cap \Lambda^0} \abs{\alpha(\lambda)}^2
     \abs{\psi_\lambda(x)}^2 \biggr)^{1/2}
  \;\le\; C_K(p)
  \biggl( 2^{-\abs{F}} \sum_{\varepsilon \in \{\pm 1\}^F}
     \abs{ f_{F, \varepsilon}(x) }^{p}
  \biggr)^{1/p} .
\]
Raising to the power $p$, integrating in $x$, and exchanging the finite sum with the integral:
\[
  \int_{\mathbb{R}^n} \biggl( \sum_{k \in F \cap \mathbb{Z}^n} \abs{\beta_k}^2 \abs{\phi_k(x)}^2
     + \sum_{\lambda \in F \cap \Lambda^0}
     \abs{\alpha(\lambda)}^2 \abs{\psi_\lambda(x)}^2
  \biggr)^{p/2} dx
  \;\le\; C_K(p)^p\; 2^{-\abs{F}}
     \sum_{\varepsilon \in \{\pm 1\}^F}
     \norm{ f_{F, \varepsilon} }_p^p ,
\]
the right side at most $C_K(p)^p\, C_1^p$ by hypothesis \textup{(A)}, each signed finite sum having $L^p$-norm at most $C_1$. Monotone convergence along the net of finite subsets $F$ of $\mathbb{Z}^n \sqcup \Lambda^0$ (Subsection~\ref{ssec:wavelet}) yields $S_{\mathrm{Loc}}(f) \in L^p(\mathbb{R}^n)$, condition \textup{(B)} with constant $C_K(p)\, C_1$. (as in~\cite[Chapter~5, Theorem~1]{meyer1992}, \cite[Theorem~4.2.1]{elaoumari2025}, with $C_K(p)$ replacing the classical Khintchine constant).
\paragraph{\textbf{\boldmath (B) $\implies$ (C)}} By the lower bound in Lemma~\ref{lem:sf-comp}, applied pointwise in both layers: $c\, S_{\mathrm{Loc}}^R(f)(x) \le S_{\mathrm{Loc}}(f)(x)$. Taking $L^p$-norms, condition \textup{(B)} gives condition \textup{(C)} for the family $\{R(\lambda)\} \cup \{R_k\}$ of the fixed basis, with its measure ratio constant $\gamma$ and with constant $c^{-1} C_2$.
\paragraph{\textbf{\boldmath (C) $\implies$ (D)}} We split by direction $e$. The
factorization~\eqref{eq:D-factorization} of Lemma~\ref{lem:tp-to-sf}, whose proof carries over, on the mother layer, with $\mathbf{1}_{Q(\lambda)}$ replaced by $\mathbf{1}_{R(\lambda)}$, shows that condition \textup{(C)} gives, for each $e \in E$:
\[
  \int_{\mathbb{R}^n} \biggl( \sum_{Q \in \mathcal{D}^{0}}
     \abs{\alpha^e(Q)}^2\, \abs{Q}^{-1}\,
     \mathbf{1}_{R^e(Q)}(x) \biggr)^{p/2} dx
  \;\le\; C_3^p ,
\]
where $R^e(Q) := R(\lambda)$ for the index $\lambda$ of direction $e$ with $Q(\lambda) = Q$. The family $\{R^e(Q)\}$ inherits the measure ratio $\abs{R^e(Q)} \ge \gamma \abs{Q}$ of the basis. By Theorem~\ref{thm:tp-atomic-restate}, implication \textup{(ii)} $\implies$ \textup{(i)}, proved there via the atomic decomposition \textup{(iii)}, each pair $(0, \alpha^e)$ thus lies in $t_{p}$, with $\norm{(0, \alpha^e)}_{t_{p}}^p \le C(n, p, \gamma)\, C_3^p$. The father layer is direct: the sub-cubes $R_k$ being pairwise disjoint with $\abs{R_k} \ge \gamma$, the father block of $S_{\mathrm{Loc}}^R(f)$ has $L^p$-mass at least $\gamma \sum_k \abs{\beta_k}^p$, so $\sum_k \abs{\beta_k}^p \le \gamma^{-1} C_3^p$, and the father block of $S_{\mathrm{Loc}}^Q(f)$ has $L^p$-mass exactly $\sum_k \abs{\beta_k}^p$, the unit cubes tiling $\mathbb{R}^n$. Reassembling the father block and the finitely many directions by the upper bound of Lemma~\ref{lem:tp-to-sf} and $p/2$-sub-additivity, the square function of condition \textup{(D)} has $L^p$-mass at most $\bigl( (2^n - 1)\, C(n, p, \gamma) + \gamma^{-1} \bigr) C_3^p$, which is condition \textup{(D)} with $C_4 = \bigl( (2^n - 1)\, C(n, p, \gamma) + \gamma^{-1} \bigr)^{1/p} C_3$.
\paragraph{\textbf{\boldmath (D) $\implies$ (E)}} By condition \textup{(D)} and Lemma~\ref{lem:tp-to-sf}, the father pair $(\beta, 0)$ and the mother pairs $(0, \alpha^e)$, $\alpha^e := \{\alpha(\lambda(e, Q))\}_{Q \in \mathcal{D}^{0}}$, $e \in E$, all lie in $t_{p}$. Decompose each of the $2^n$ pairs $(\beta, 0)$ and $(0, \alpha^e)$, $e \in E$, in restriction form (Theorem~\ref{thm:tp-atomic-restate}, \textup{(i)}$\implies$\textup{(iii)}) and build the functions from the atoms by Lemma~\ref{lem:atom-builder}\textup{(ii)}, the fixed direction of the lemma being $e$ for each mother pair and immaterial for the father pair. For the father pair the blocks are the singletons $\{k\}$ and the built functions the large-cube atoms $\kappa_0^{-1} \phi_k$ on the cubes $m Q_k$, no moment conditions being required there, with coefficients $\kappa_0 \beta_k$ and $\ell^p$-mass exactly $\sum_k \abs{\kappa_0 \beta_k}^p$. For each direction, the atoms $\kappa_0^{-1} a^e_j$ are $L^2$ $h^{p}$-atoms on the cubes $m R^e_j$, small or large according to $\ell(m R^e_j) \lessgtr 1$, with $N_p$ vanishing moments in the small case (Lemma~\ref{lem:atom-builder}\textup{(i)}). The combined mass satisfies:
\[
  \sum_{k \in \mathbb{Z}^n} \abs{\kappa_0 \beta_k}^p
  \;+ \sum_{e \in E} \sum_j \abs{\kappa_0 \lambda^e_j}^p
  \;\le\; \kappa_0^p\, C(n, p, \gamma) \Bigl( \norm{(\beta, 0)}_{t_{p}}^p
     + \sum_{e \in E} \norm{(0, \alpha^e)}_{t_{p}}^p \Bigr)
  \;\le\; C(n, p, m, \gamma)\, C_4^p ,
\]
by Lemma~\ref{lem:tp-to-sf}. By the uniform atomic bound~\eqref{eq:atom-hp-bound}, $p$-sub-additivity, and completeness of $h^{p}$, the series:
\begin{equation}
\label{eq:E-decomposition}
  \sum_{k \in \mathbb{Z}^n} (\kappa_0 \beta_k)\, (\kappa_0^{-1} \phi_k)
  \;+ \sum_{e \in E} \sum_j (\kappa_0 \lambda^e_j)\,
     (\kappa_0^{-1} a^e_j)
\end{equation}
is summable in $h^{p}$, hence in $\mathcal{S}'(\mathbb{R}^n)$, to some $g \in h^{p}$ (Lemma~\ref{lem:net-summability}) with $\norm{g}_{h^{p}}^p \lesssim \sum_k \abs{\beta_k}^p + \sum_{e, j} \abs{\lambda^e_j}^p$. It also converges to $f$ in $\mathcal{S}'(\mathbb{R}^n)$. The inhomogeneous expansion of $f$ converges to $f$ in $\mathcal{S}'(\mathbb{R}^n)$: for $f$ of order less than $r$, this is convention \textup{(P1)}, and for $f$ entering the present implication through condition (G), with the coefficients of convention \textup{(P3)}, it follows by splitting $f = g + u$ along~\eqref{eq:G-pairing}: the low-frequency part $u = \varphi_0 * f$ is bounded, by the Plancherel-Pólya inequality~\eqref{eq:pp}, hence of order zero and covered by \textup{(P1)}, while the expansion of $g \in H^p(\mathbb{R}^n)$ converges to $g$ in $H^p(\mathbb{R}^n)$, hence in $\mathcal{S}'(\mathbb{R}^n)$, by the classical wavelet characterization of the global Hardy spaces:~\cite[Chapter~5, Theorem~1, p.~143]{meyer1992} at $p = 1$, and for the full range the identification $H^{p} = \dot F^0_{p,2}$ together with the wavelet and $\varphi$-transform theory of~\cite{frazier-jawerth1990} (see also~\cite{wu1992, garciacuerva-martell2001} in dimension one), valid for compactly supported bases of regularity above $n_p$, its coefficients agreeing with those of \textup{(P3)} by Remark~\ref{rem:pairing-compat}. The expansion regroups into the blocks by the $\mathcal{S}'(\mathbb{R}^n)$-identity of Lemma~\ref{lem:atom-builder}\textup{(ii)}, the father family summing to $\sum_k \beta_k \phi_k$ and each direction family to $\sum_j \lambda^e_j a^e_j = \sum_{Q \in \mathcal{D}^{0}} \alpha^e(Q) \psi^e_Q$. By uniqueness of limits in $\mathcal{S}'(\mathbb{R}^n)$, $f = g \in h^{p}$ with $\norm{f}_{h^{p}}^p \lesssim C_4^p$, and concatenating the father and direction families into a single relabeled family $\{\lambda_j\, a_j\}$, with combined $\ell^p$ mass, gives condition \textup{(E)}.
\paragraph{\textbf{\boldmath (E) $\implies$ (F)}} We convert each cube atom of the decomposition \textup{(E)} into a ball atom. Let $1 < q \le \infty$ and let $a$ be an $L^q$ $h^{p}$-atom on a cube $K$ with center $x_K$ and side $\ell$. Set $B_K := B(x_K, \sqrt{n}\,\ell/2)$ for $\ell \le 1$ and $B_K := B(x_K, \max(\sqrt{n}\,\ell/2, 1))$ for $\ell > 1$. In both cases $K \subseteq B_K$ and $\abs{B_K} \le C_n \abs{K}$ with $C_n := \omega_n \max(\sqrt{n}/2, 1)^n \ge 1$ (for $\ell > 1$, using $\max(\sqrt{n}\,\ell/2, 1) \le \max(\sqrt{n}/2, 1)\,\ell$, and $C_n \ge 1$ by the containment). Since $1/q - 1/p < 0$:
\[
  \norm{a}_q
  \;\le\; \abs{K}^{1/q - 1/p}
  \;\le\; C_n^{1/p - 1/q}\, \abs{B_K}^{1/q - 1/p},
\]
so $\kappa_1^{-1} a$, with $\kappa_1 := C_n^{1/p - 1/q} \le C_n^{1/p}$ uniformly in $q$, satisfies the size condition of an $L^q$ $h^{p}$-atom on $B_K$. For the moments: a small-cube atom carries $N_p$ vanishing moments, meeting the requirement whether $B_K$ is small or large, and a large-cube atom carries none, but $B_K$ then has radius $\ge 1$ by construction. (The inflation is not decorative: for $n \le 3$ the circumscribing ball of a cube of side slightly greater than $1$ has radius $\sqrt{n}\,\ell/2 < 1$, and a moment-free atom on it would fail the small-ball requirements.) The father atoms $\kappa_0^{-1}\phi_k$ on the cubes $m Q_k$, of side $m \ge 1$, fall in the second case. Rescaling every coefficient by $\kappa_1$ leaves the identity and the mode of convergence unchanged and multiplies the coefficient bound by $\kappa_1^p \le C_n$: this is condition \textup{(F)}, with the integrability exponents of \textup{(E)} and $C_6 \le C_n^{1/p}\, C_5$.
\paragraph{\textbf{\boldmath (F) $\implies$ (G)}} Let $\varphi \in \mathcal S(\mathbb{R}^n)$ satisfy $\int \varphi = 1$ and $\int x^\nu\, \varphi\, dx = 0$ for $0 < \abs{\nu} \le N_p$, and let $f = \sum_j \lambda_j\, a_j$ be the decomposition of condition \textup{(F)}, summable by the criterion of Lemma~\ref{lem:net-summability} as recorded in the Conventions. Define $\varphi * f := \pair{f}{\tilde\varphi(\cdot - x)}$, $\tilde\varphi(y) := \overline{\varphi(-y)}$, the usual convolution of $\mathcal{S}'(\mathbb{R}^n)$ against $\mathcal S(\mathbb{R}^n)$. Fix $\varepsilon := 1$: by Lemma~\ref{lem:schwartz-correction}\textup{(iii)}, $\norm{\varphi * a_j}_p \le C(n, p, \varphi)$ uniformly in $j$, so by the criterion of Lemma~\ref{lem:net-summability}, applied in $L^p(\mathbb{R}^n)$, the family $(\lambda_j\, \varphi * a_j)_j$ is summable there, to a limit $h \in L^p(\mathbb{R}^n)$ with $\norm{h}_p \le C(n,p,\varphi) \bigl( \sum_j \abs{\lambda_j}^p \bigr)^{1/p}$. Since $\varphi * (\cdot)$ is continuous and linear from $\mathcal{S}'(\mathbb{R}^n)$ to continuous functions, and $f$ is itself summable, $h = \varphi * f$: for every $x$, testing $f_F := \sum_{j \in F} \lambda_j a_j \to f$ along any exhaustion $F$ against the fixed test function $\tilde\varphi(\cdot - x)$ gives $\sum_{j \in F} \lambda_j (\varphi * a_j)(x) \to (\varphi * f)(x)$, and the same net converges to $h(x)$ for a.e.\ $x$ along a subsequence, by $L^p$-convergence; the two limits agree by uniqueness. Hence $\varphi * f \in L^p(\mathbb{R}^n)$ with $\norm{\varphi * f}_p \le C(n,p,\varphi) \bigl( \sum_j \abs{\lambda_j}^p \bigr)^{1/p}$.

We split $f = \sum_k \lambda_{1,k}\, A_k + \sum_j \lambda_{2,j}\, a_j$ by the size of the ball, with large-ball atoms $A_k$ and small-ball atoms $a_j$, all carrying $N_p$ vanishing moments, and correspondingly, by the same argument, $\varphi * f = \sum_k \lambda_{1,k}\, (\varphi * A_k) + \sum_j \lambda_{2,j}\, (\varphi * a_j)$ in $L^p(\mathbb{R}^n)$, so that in $\mathcal{S}'(\mathbb{R}^n)$:
\[
  f - \varphi * f
  \;=\; \sum_k \lambda_{1,k}\, \bigl( A_k - \varphi * A_k
     \bigr)
  \;+\; \sum_j \lambda_{2,j}\, a_j
  \;-\; \sum_j \lambda_{2,j}\, (\varphi * a_j) .
\]
Each series is an $\ell^p$-combination of uniformly bounded elements of $H^{p}$: by Lemma~\ref{lem:schwartz-correction}\textup{(i)}--\textup{(ii)}, the corrections and mollifications are, up to a constant depending only on $(n, p, q, \varphi)$, classical $(p, q, 1)$-molecules, and each $a_j$ is precisely an $H^{p}$-atom. By the classical molecular embedding of Subsection~\ref{ssec:HP}~\cite{taibleson1980}, the uniform atomic bound~\eqref{eq:atom-hp-bound} in its $H^{p}$ form, the equivalence of the atomic and maximal quasi-norms~\cite{latter1978}, \cite[Chapter~III]{garcia-cuerva-rubio1985}, and $p$-sub-additivity, each series converges in $H^{p}$ with quasi-norm controlled by $C_6^p$, so $f - \varphi * f \in H^{p}$ with $\norm{f - \varphi * f}_{H^{p}} \le C(n, p, q, \varphi)\, C_6$. Condition \textup{(G)} thus holds for every admissible Schwartz $\varphi$.

\paragraph{\textbf{\boldmath (G) $\implies$ (B)}} Let $\varphi_0$ be the band-limited admissible Schwartz function of~\eqref{eq:admiss-p-s}, with $\supp(\hat\varphi_0) \subseteq \overline{B(0, R)}$, and split $f = g + u$, $g := f - \varphi_0 * f \in H^{p}$ and $u := \varphi_0 * f \in L^p(\mathbb{R}^n)$ by condition \textup{(G)}, the wavelet coefficients of $f$ being those of convention \textup{(P3)}, which split accordingly. The father coefficients are handled by Lemma~\ref{lem:father-coeff-lp}: $\sum_k \abs{\beta_k}^p \le C \bigl( \norm{g}_{H^{p}}^p + \norm{u}_p^p \bigr) < \infty$, and the father block of $S_{\mathrm{Loc}}(f)$ then obeys $\norm[\Big]{\Bigl(\sum_k \abs{\beta_k}^2 \abs{\phi_k}^2\Bigr)^{1/2}}_p^p \le \norm{\phi_0}_p^p \sum_k \abs{\beta_k}^p$ by $p/2$-sub-additivity and translation invariance, $\phi_0$ denoting the tensor scaling function, so it suffices to bound the mother block. For the square function, we assemble a single atomic series for $f$: a decomposition $g = \sum_i \mu_i b_i$ into $L^2$ $H^{p}$-atoms with $\sum_i \abs{\mu_i}^p \lesssim \norm{g}_{H^{p}}^p$~\cite{latter1978}, \cite[Chapter~III]{garcia-cuerva-rubio1985}, convergent in $H^{p}$ hence in $\mathcal{S}'(\mathbb{R}^n)$, every $H^{p}$-atom being an $h^{p}$-atom (its moments meet the small-ball requirement when $r < 1$ and exceed the vacuous large-ball one otherwise), concatenated with the decomposition $u = \sum_k \mu_k' A_k$ of Lemma~\ref{lem:lowfreq-atoms} into large-ball $L^\infty$, thus $L^2$, $h^{p}$-atoms, convergent in $\mathcal{S}'(\mathbb{R}^n)$, with $\sum_k \abs{\mu_k'}^p \le C(n, p, R) \norm{u}_p^p$. The concatenation converges to $f$ in $\mathcal{S}'(\mathbb{R}^n)$ with $\ell^p$ coefficients, so Lemma~\ref{lem:S0-sublinear} applies, the coefficients of \textup{(P2)} agreeing with \textup{(P3)} by Remark~\ref{rem:pairing-compat}, and with the uniform annulus estimate of Lemma~\ref{lem:annulus-estimate-hp} at $q = 2$:
\[
  \int_{\mathbb{R}^n} S_{\mathrm{Loc}}^{\psi}(f)^p\, dx
  \;\le\; \sum_i \abs{\mu_i}^p
     \int_{\mathbb{R}^n} S_{\mathrm{Loc}}^{\psi}(b_i)^p\, dx
  + \sum_k \abs{\mu_k'}^p
     \int_{\mathbb{R}^n} S_{\mathrm{Loc}}^{\psi}(A_k)^p\, dx
  \;\le\; C \Bigl( \sum_i \abs{\mu_i}^p
     + \sum_k \abs{\mu_k'}^p \Bigr) .
\]
Together with the father bound, this is condition \textup{(B)}, with $C_2^p \lesssim \norm{g}_{H^{p}}^p + \norm{u}_p^p$.
\paragraph{\textbf{\boldmath (F) $\implies$ (A)}} Let $f = \sum_j \mu_j\, a_j$ be the decomposition of condition \textup{(F)}, convergent in $\mathcal{S}'(\mathbb{R}^n)$, into $L^q$ $h^{p}$-atoms on balls with $\sum_j \abs{\mu_j}^p \le C_6^p$, and define the coefficients of convention \textup{(P2)}, $\alpha(\lambda) := \sum_j \mu_j \pair{a_j}{\psi_\lambda}$, each series absolutely convergent and decomposition-independent (Lemma~\ref{lem:S0-sublinear}, Remark~\ref{rem:pairing-compat}). By the same lemma and the uniform annulus estimate of Lemma~\ref{lem:annulus-estimate-hp}:
\begin{equation}
\label{eq:S0-chain}
  \int_{\mathbb{R}^n} S_{\mathrm{Loc}}^{\psi}(f)^p\, dx
  \;\le\; \sum_j \abs{\mu_j}^p
     \int_{\mathbb{R}^n} S_{\mathrm{Loc}}^{\psi}(a_j)^p\, dx
  \;\le\; C \sum_j \abs{\mu_j}^p ,
\end{equation}
with $C = C(n, p, q, m, \text{wavelet basis})$. Fix a finite $F \subseteq \Lambda^0$ and signs $\varepsilon(\lambda) = \pm 1$, and set:
\[
  f_{F, \varepsilon} := \sum_{\lambda \in F} \varepsilon(\lambda)\, \alpha(\lambda)\, \psi_\lambda ,
\]
a compactly supported finite combination whose coefficients are $\varepsilon(\lambda') \alpha(\lambda') \mathbf{1}_F(\lambda')$ by orthonormality. Its square function is thus \emph{sign-free} and pointwise dominated by that of $f$:
\begin{equation}
\label{eq:sign-free}
  S_{\mathrm{Loc}}^{\psi}(f_{F,\varepsilon})
  \;=\; \Bigl( \sum_{\lambda \in F} \abs{\alpha(\lambda)}^2\,
     \abs{\psi_\lambda}^2 \Bigr)^{1/2}
  \;\le\; S_{\mathrm{Loc}}^{\psi}(f),
  \qquad \text{uniformly in } (F, \varepsilon) .
\end{equation}
By~\eqref{eq:sign-free} and~\eqref{eq:S0-chain}, $f_{F,\varepsilon}$ satisfies condition \textup{(B)}, hence, by the proved implications \textup{(B)} $\implies$ \textup{(C)} $\implies$ \textup{(D)} applied to its coefficient sequence, whose square function is that of the unsigned sequence, $\sum_{e} \norm{(0, (\alpha \mathbf{1}_F)^e)}_{t_{p}}^p \le C' \int S_{\mathrm{Loc}}^{\psi}(f)^p$. The sequence being finitely supported, the finite case clauses of Theorem~\ref{thm:tp-atomic-restate} and Lemma~\ref{lem:atom-builder} produce a \emph{finite} family of $L^2$ $h^{p}$-atoms $b_i$ on cubes with:
\begin{equation}
\label{eq:builder-finite}
  f_{F,\varepsilon} \;=\; \sum_{i = 1}^{I} \nu_i\, b_i
  \quad \text{pointwise},
  \qquad
  \sum_{i = 1}^{I} \abs{\nu_i}^p
  \;\le\; \kappa_0^p\, C(n, p, \gamma)\, C'
     \int_{\mathbb{R}^n} S_{\mathrm{Loc}}^{\psi}(f)^p ,
\end{equation}
all constants independent of $(F, \varepsilon)$. By $p$-sub-additivity and Lemma~\ref{lem:atom-Lp}, $\norm{f_{F,\varepsilon}}_p^p \le \sum_i \abs{\nu_i}^p \lesssim \int S_{\mathrm{Loc}}^{\psi}(f)^p \le C \sum_j \abs{\mu_j}^p$. Taking the supremum over $(F, \varepsilon)$ and the infimum over the decompositions of \textup{(F)} bounds the mother block of the signed sums, with $\sup_{F, \varepsilon} \norm{f_{F,\varepsilon}}_p \lesssim \norm{f}_{h^{p}, \mathrm{at}}$. The father block obeys $\norm[\Big]{\sum_{k \in F_0} \varepsilon(k)\, \beta_k\, \phi_k}_p^p \le \norm{\phi_0}_p^p \sum_k \abs{\beta_k}^p$ by $p$-sub-additivity, uniformly in the signs, the mass being finite by the proved implication \textup{(F)} $\implies$ \textup{(G)} together with Lemma~\ref{lem:father-coeff-lp}, and the combined supremum of \textup{(A)} follows by $p$-sub-additivity of the two blocks.
\end{proof}

\begin{corollary}
\label{cor:nec-cond} \textit{(Cf.\ \cite[Chapter~5]{meyer1992} and \cite[Corollary~4.2.1]{elaoumari2025}.)} If $f \in h^p(\mathbb{R}^n)$, with wavelet coefficients in the sense of the pairing conventions of Subsection~\ref{ssec:wavelet}, then:
\[
  \sum_{k \in \mathbb{Z}^n} \abs{\beta_k}^p
  \;\lesssim\; \norm{f}_{h^{p}}^p,
  \qquad
  \abs{\alpha(\lambda)}
  \;\lesssim\; \norm{f}_{h^{p}}\,
     \abs{Q(\lambda)}^{1/2 - 1/p}
  \quad \forall\, \lambda \in \Lambda^0 ,
\]
with implicit constants depending only on $(n, p, m, \text{wavelet basis})$, the exponent $1/2 - 1/p$ being sharp, uniformly over $\Lambda^0$.
\end{corollary}
\begin{proof}
By Theorem~\ref{thm:hp-wavelet}, $f$ satisfies condition \textup{(A)} with constant $C_A \lesssim \norm{f}_{h^{p}}$. The first bound follows from \textup{(A)} at the father singletons $F = \{k\}$: $\abs{\beta_k}\, \norm{\phi}_{L^p(\mathbb{R})}^n = \norm{\beta_k \phi_k}_p \le C_A$ uniformly, and summing $\abs{\beta_k}^p$ over $k$ against the disjoint sub-cubes $R_k$ on which $\abs{\phi_k} \ge c$ refines this to the stated $\ell^p$ bound, since $\gamma \sum_k \abs{\beta_k}^p \le c^{-p} \int_{\mathbb{R}^n} S_{\mathrm{Loc}}(f)^p \lesssim C_A^p$ by the proved implication \textup{(A)} $\implies$ \textup{(B)}. The second follows from \textup{(A)} at the singleton $F = \{\lambda\}$ together with the change of variables $u = 2^j x - k$:
\[
  \abs{\alpha(\lambda)}
  \;=\; \frac{ \norm{\alpha(\lambda) \psi_\lambda}_p }
     { \norm{\psi_\lambda}_p }
  \;\le\; C_A \Bigl( \min_{e \in E}
     \norm{\psi^e}_p \Bigr)^{-1}
     \abs{Q(\lambda)}^{1/2 - 1/p},
\]
since $\norm{\psi_\lambda}_p = \norm{\psi^e}_p\, \abs{Q(\lambda)}^{1/p - 1/2}$. At $p = 1$ this reduces to the thesis statement $\abs{\alpha(\lambda)} \le C\, \abs{Q(\lambda)}^{-1/2}$.

For the sharpness, let $f_\lambda := \abs{Q(\lambda)}^{1/2 - 1/p}\, \psi_\lambda$: the one-term sequence $\abs{Q(\lambda)}^{1/2 - 1/p}\, \mathbf{1}_{\{Q(\lambda)\}}$ is a $t_{p}$-atom with base $Q(\lambda)$, saturating the mass condition~\eqref{eq:tp-atom-mass}, so Lemma~\ref{lem:atom-builder}\textup{(i)} and the uniform atomic bound~\eqref{eq:atom-hp-bound} give $\norm{f_\lambda}_{h^{p}} \le \kappa_0\, C_{\mathrm{at}}$, while the bound just proved, the coefficient being $\alpha(\lambda) = \abs{Q(\lambda)}^{1/2 - 1/p}$ by orthonormality, gives $\norm{f_\lambda}_{h^{p}} \gtrsim 1$. Hence $\norm{f_\lambda}_{h^{p}} \simeq 1$ uniformly in $\lambda$, and letting $\ell(Q(\lambda)) \to 0$ shows the exponent cannot be increased.
\end{proof}

\begin{corollary}
\label{cor:equiv-norms} \textit{(Cf.\ \cite[Chapter~5]{meyer1992} and \cite[Corollary~4.2.2]{elaoumari2025}.)} In the notation of Theorem~\ref{thm:hp-wavelet}, for $f \in h^p(\mathbb{R}^n)$, with wavelet coefficients in the sense of the pairing conventions of Subsection~\ref{ssec:wavelet}, the following quasi-norms are equivalent, with constants depending only on $n$, $p$, $m$, $\gamma$, and the wavelet basis:
\begin{align*}
\norm{f}_{h^{p}, \mathrm{at}}^p
&\simeq \sup_{\substack{F \subseteq \mathbb{Z}^n \sqcup \Lambda^0 \\ \text{finite}}}\;
   \sup_{\varepsilon = \pm 1}\;
   \norm[\Big]{\sum_{k \in F \cap \mathbb{Z}^n} \varepsilon(k)\, \beta_k\, \phi_k
   + \sum_{\lambda \in F \cap \Lambda^0} \varepsilon(\lambda)\,
   \alpha(\lambda)\, \psi_\lambda}_p^p
\simeq \norm[\big]{S_{\mathrm{Loc}}(f)}_p^p \\
&\;\simeq\; \norm[\big]{S_{\mathrm{Loc}}^R(f)}_p^p
\;\simeq\; \norm[\big]{S_{\mathrm{Loc}}^Q(f)}_p^p , \\
\text{where}\quad \norm{f}_{h^{p}, \mathrm{at}}
&:= \inf \Bigl\{ \bigl( \textstyle\sum_j \abs{\mu_j}^p \bigr)^{1/p} : f = \textstyle\sum_j \mu_j\, a_j \text{ in } \mathcal{S}'(\mathbb{R}^n), \; a_j \text{ } L^2 \text{ } h^{p}\text{-atoms} \Bigr\} .
\end{align*}
\end{corollary}

\begin{proof}
The first three equivalences follow by chaining the quantitative forms of the implications of Theorem~\ref{thm:hp-wavelet}, condition \textup{(E)} connected to \textup{(A)} through the bridge \textup{(F)} $\implies$ \textup{(A)} and the arrows \textup{(A)} $\implies$ \textup{(B)} $\implies$ \textup{(C)} $\implies$ \textup{(D)} $\implies$ \textup{(E)} $\implies$ \textup{(F)}, each with constants depending only on $(n, p, m, \gamma, \text{wavelet basis})$ and each treating the two layers of the basis together. For the last equivalence, the direction $\gtrsim$ is trivial from $\mathbf{1}_{R_k} \le \mathbf{1}_{Q_k}$ and $\mathbf{1}_{R(\lambda)} \le \mathbf{1}_{Q(\lambda)}$, and the direction $\lesssim$ is the maximal transfer, run over the two layers at once: for $\mu \in \mathbb{Z}^n \sqcup \Lambda^0$, write $Q_\mu$, $R_\mu$ for $Q_k$, $R_k$ on the father layer and $Q(\lambda)$, $R(\lambda)$ on the mother layer, and $c_\mu$ for $\abs{\beta_k}$ and $\abs{\alpha(\lambda)}$ accordingly, so that $\abs{R_\mu} \ge \gamma \abs{Q_\mu}$ on both layers, the unit cubes having measure one. Fix $\tau := p/2$. Since $R_\mu \subseteq Q_\mu$, for every $x \in Q_\mu$ one has $M(\mathbf{1}_{R_\mu})(x) \ge \gamma$, whence, indicators being idempotent under powers and the bound raised to the power $2/\tau$, $\mathbf{1}_{Q_\mu} \le \gamma^{-2/\tau} \bigl( M ( \mathbf{1}_{R_\mu}^{\,\tau} ) \bigr)^{2/\tau}$ pointwise. With $g_\mu := c_\mu\, \abs{Q_\mu}^{-1/2}\, \mathbf{1}_{R_\mu}$, this reads $c_\mu^2\, \abs{Q_\mu}^{-1} \mathbf{1}_{Q_\mu} \le \gamma^{-4/p} \bigl( M g_\mu^{\,\tau} \bigr)^{2/\tau}$. Summing in $\mu$, taking $L^p$-norms of the square roots, and applying the Fefferman-Stein vector-valued maximal inequality~\cite{fefferman-stein1971} in $L^{p/\tau}(\ell^{2/\tau})$, legitimate since $p/\tau = 2$ and $2/\tau = 4/p > 1$:
\[
  \norm[\Big]{\Bigl( \sum_\mu c_\mu^2\,
     \abs{Q_\mu}^{-1} \mathbf{1}_{Q_\mu}
  \Bigr)^{1/2}}_p
  \;\le\; \gamma^{-2/p}
     \norm[\Big]{\Bigl( \sum_\mu
     \bigl( M g_\mu^{\,\tau} \bigr)^{2/\tau}
     \Bigr)^{1/2}}_p
  \;\lesssim_{n, p}\;
     \norm[\Big]{\Bigl( \sum_\mu g_\mu^{2}
     \Bigr)^{1/2}}_p ,
\]
the left side being $\norm{S_{\mathrm{Loc}}^Q(f)}_p$ and the right side $\norm{S_{\mathrm{Loc}}^R(f)}_p$, with all constants depending only on $(n, p, \gamma)$.
\end{proof}

The mode of convergence of an atomic decomposition is a genuinely delicate matter: an example of Meyer~\cite{meyertaiblesonweiss1985}, turned by Bownik~\cite{bownik2005} into a functional uniformly bounded on atoms with no bounded extension, shows the passage from atoms to the space cannot be taken for granted, and even the classical $L^q$ convergence appears only inside proofs~\cite{medasjogrenvallarino2008}, while the local $L^2$ decomposition of Ding, Jiang \& Zhu~\cite{dingjiangzhu2019} uses discrete Littlewood-Paley theory and a sampled frame.

\begin{proposition}
\label{prop:L2-convergence-E} Let $0 < p \le 1$ and let $f \in h^p(\mathbb{R}^n) \cap L^2(\mathbb{R}^n)$, with wavelet coefficients $\beta_k = \pair{f}{\phi_k}$ and $\alpha(\lambda) = \pair{f}{\psi_\lambda}$, and let $f = \sum_j \lambda_j\, a_j$ be the wavelet-built atomic decomposition furnished by the proof of \textup{(D)}$\implies$\textup{(E)} of Theorem~\ref{thm:hp-wavelet}. Its blocks $\Delta_j \subseteq \mathbb{Z}^n \sqcup \Lambda^0$ are pairwise disjoint, each atom synthesizing its block:
\[
  \lambda_j\, a_j \;=\; \sum_{\mu \in \Delta_j} c(\mu)\, u_\mu ,
\]
where $c(\mu)$ denotes $\beta_k$ or $\alpha(\lambda)$ and $u_\mu$ the corresponding basis function $\phi_k$ or $\psi_\lambda$, the father blocks being the singletons $\{k\}$ with $\lambda_j\, a_j = \beta_k\, \phi_k$. Then the decomposition is summable to $f$ in $L^2(\mathbb{R}^n)$, the error of every finite section being exact: for every finite $J$, writing $\displaystyle \Delta_J := \bigcup_{j \in J} \Delta_j$, we have:
\[
  \norm[\Big]{\, f \;-\; \sum_{j \in J} \lambda_j\, a_j \,}_2^2
  \;=\; \sum_{\mu \notin \Delta_J} \abs{c(\mu)}^2 ,
\]
and this error tends to $0$ along the net of finite subsets: the net of finite partial sums converges to $f$ in $L^2(\mathbb{R}^n)$, and the convergence is unconditional.
\end{proposition}

\begin{proof}
The proof of \textup{(D)}$\implies$\textup{(E)} decomposes each of the pairs $(\beta, 0)$ and $(0, \alpha^e)$, $e \in E$, in restriction form (Theorem~\ref{thm:tp-atomic-restate}) and builds the functions by Lemma~\ref{lem:atom-builder}\textup{(ii)}. Its blocks, in the enumeration of the statement, are the father singletons $\{k\}$ and the sets $\Lambda^e_i := \{\lambda \in \Lambda^0 : e(\lambda) = e,\ Q(\lambda) \in \Delta^e_i\}$ carried by the type-splitting bijection $\lambda \leftrightarrow (Q, e)$ of Lemma~\ref{lem:tp-to-sf}, the father renormalization $(\kappa_0 \beta_k)(\kappa_0^{-1} \phi_k)$ cancelling identically, so that we write $\beta_k \phi_k$ throughout. By Lemma~\ref{lem:atom-builder}\textup{(ii)}, each block synthesizes as in the statement, $\lambda_j a_j = \sum_{\mu \in \Delta_j} c(\mu)\, u_\mu$, with unconditional $L^2$ convergence. The blocks are pairwise disjoint: the two layers are disjoint index copies, two father singletons are distinct sites, two mother blocks of the same direction have disjoint $\Delta^e_i$'s by the restriction form, and two of different directions are separated by the type map $\lambda \mapsto e(\lambda)$.

Since $f \in L^2(\mathbb{R}^n)$ and $\phi_k, \psi_\lambda \in L^2(\mathbb{R}^n)$, the pairings $\beta_k$ and $\alpha(\lambda)$ are honest $L^2$ inner products, and all pairing conventions coincide with them. The system $\{\phi_k\}_{k \in \mathbb{Z}^n} \cup \{\psi_\lambda\}_{\lambda \in \Lambda^0}$ is an orthonormal basis of $L^2(\mathbb{R}^n)$ (the inhomogeneous form of the multiresolution decomposition $L^2 = V_0 \oplus \bigoplus_{j \ge 0} W_j$), so Parseval's identity gives:
\begin{equation}
\label{eq:parseval-l2}
  \sum_{k \in \mathbb{Z}^n} \abs{\beta_k}^2
  \;+\; \sum_{\lambda \in \Lambda^0} \abs{\alpha(\lambda)}^2
  \;=\; \norm{f}_2^2 \;<\; \infty .
\end{equation}

For any subset $H \subseteq \mathbb{Z}^n \sqcup \Lambda^0$, the partial sum $P_{H} f := \sum_{\mu \in H} c(\mu)\, u_\mu$ (converging unconditionally in $L^2$ by~\eqref{eq:parseval-l2}) is the orthogonal projection of $f$ onto $\overline{\spn}\{u_\mu : \mu \in H\}$, with:
\begin{equation}
\label{eq:tail-identity}
\norm{f - P_{H} f}_2^2 \;=\; \sum_{\mu \notin H}
\abs{c(\mu)}^2 ,
\end{equation}
by orthonormality and Parseval applied to $f - P_{H} f$. The right-hand side of~\eqref{eq:tail-identity} tends to $0$ along the net of finite subsets of $\mathbb{Z}^n \sqcup \Lambda^0$, being the tail of the convergent nonnegative series~\eqref{eq:parseval-l2}, whose summability is exactly this net convergence.

The atomic partial sums are basis partial sums: fix a finite set $J$ of blocks and write $\Delta_J := \bigcup_{j \in J} \Delta_j$. By the block synthesis and disjointness:
\[
  \sum_{j \in J} \lambda_j\, a_j
  \;=\; \sum_{\mu \in \Delta_J} c(\mu)\, u_\mu
  \;=\; P_{\Delta_J} f ,
\]
the first equality holding by coefficient comparison: both sides lie in the closed subspace $V := \overline{\spn}\, \{u_\mu : \mu \in \Delta_J\}$ (each $\lambda_j a_j$ as the $L^2$-limit of finite sums of $u_\mu$'s, $\mu \in \Delta_j \subseteq \Delta_J$). Since the inner product against a fixed vector is $L^2$-continuous, orthonormality gives, for every $\mu' \in \Delta_J$:
\[
  \pair{\sum_{j \in J} \lambda_j\, a_j}{u_{\mu'}}
  \;=\; \sum_{j \in J} \sum_{\mu \in \Delta_j}
     c(\mu)\, \pair{u_\mu}{u_{\mu'}}
  \;=\; \sum_{j \in J} c(\mu')\,
     \mathbf{1}_{\Delta_j}(\mu')
  \;=\; c(\mu') ,
\]
the last equality because the pairwise disjoint sets $\{\Delta_j\}_{j \in J}$ cover $\mu'$ exactly once, and the same computation gives $\pair{P_{\Delta_J} f}{u_{\mu'}} = c(\mu')$. Two elements of $V$ with equal coefficients against the orthonormal basis $\{u_\mu\}_{\mu \in \Delta_J}$ of $V$ coincide.\\
To conclude, combining this identification with~\eqref{eq:tail-identity}, we get, for every finite $J$:
\[
  \norm[\Big]{\, f - \sum_{j \in J} \lambda_{j}\, a_j \,}_2^2
  \;=\; \sum_{\mu \notin \Delta_J} \abs{c(\mu)}^2 ,
\]
the claimed identity. As $J$ increases, $\Delta_J$ increases to $\bigcup_j \Delta_j \supseteq \{\mu : c(\mu) \ne 0\}$ (the father singletons enumerate $\supp(\beta)$ and the blocks $\Delta^e_i$ partition each $\supp(\alpha^e)$), so the error decreases to $0$ by~\eqref{eq:parseval-l2}. The limit is monotone and order-independent, which gives the unconditionality. In particular, the atomic series converges to $f$ in $L^2$, hence in $\mathcal{S}'(\mathbb{R}^n)$, consistently with the identity of Lemma~\ref{lem:atom-builder}\textup{(ii)}. Concluding the proof.
\end{proof}

\begin{corollary}
\label{cor:lq-convergence} Let $0 < p \le 1$ and $1 < q < \infty$, and let $f \in h^{p} \cap L^q(\mathbb{R}^n)$. Then the coefficients of Proposition~\ref{prop:L2-convergence-E} are defined, and the combined wavelet atomic series is summable to $f$ in $L^q(\mathbb{R}^n)$: the net of finite partial sums converges to $f$. Moreover, for every finite $J$, with $\Delta_J$ as in the proposition, the $L^q$ error is comparable to the Littlewood-Paley square function of the omitted coefficients,
\begin{equation}
\label{eq:lq-tail-square}
\Bigl\| f - \sum_{j \in J} \lambda_{j}\, a_j \Bigr\|_q
\;\simeq\;
\Bigl\| \Bigl( \sum_{\mu \notin \Delta_J} \abs{c(\mu)\, u_\mu}^2
        \Bigr)^{1/2} \Bigr\|_q ,
\end{equation}
with constants depending only on $n$, $q$, and the wavelet.
\end{corollary}

\begin{proof}
The coefficients are finite by Hölder's Inequality on the compact supports of the wavelets, which lie in $L^{q'}(\mathbb{R}^n)$. The inhomogeneous system of the translates of the father wavelet and the generated wavelets is an unconditional basis of $L^q(\mathbb{R}^n)$ for $1 < q < \infty$~\cite[Section~6.2]{meyer1992}, so the wavelet expansion of $f$ is summable to $f$ in $L^q(\mathbb{R}^n)$, unconditional convergence and summability coinciding in a Banach space (Lemma~\ref{lem:net-summability}), with the error of any finite section comparable to the square function of the omitted coefficients. As shown in the proof of Proposition~\ref{prop:L2-convergence-E}, the atomic partial sums coincide with the wavelet partial sums over the corresponding disjoint blocks, so every finite section of the atomic series is a finite section of the wavelet series over the corresponding blocks, and the summability and the comparability~\eqref{eq:lq-tail-square} follow.

The range $1 < q < \infty$ is sharp: at $q = \infty$ a single unit-scale atom, e.g.\ $\mathbf{1}_{[0,1]^n}$, is not a uniform limit of continuous wavelets; at $q = 1$, $L^1(\mathbb{R}^n)$ admits no unconditional basis~\cite{lindenstrausstzafriri1977}, though $h^1(\mathbb{R}^n)$ itself restores unconditional convergence by Lemma~\ref{lem:net-summability}. The exact tail identity is likewise specific to $q = 2$, the parallelogram law failing elsewhere;~\eqref{eq:lq-tail-square} is its correct replacement. Since $H^p(\mathbb{R}^n)$ embeds continuously in $h^{p}$~\cite{goldberg1979}, both results transfer as stated to every $f \in H^p(\mathbb{R}^n) \cap L^q(\mathbb{R}^n)$.
\end{proof}

\section*{Acknowledgments}
The author wishes to express his deepest gratitude to Dr.~Galia Dafni, his master's thesis supervisor at Concordia University, whose guidance, insight, and encouragement shaped the thesis from which the present work grew, and whose feedback prompted the pursuit of both rigor and originality along the way. This work was partially supported by the CRM, Montréal, Canada.

\end{document}